\let\origvec\vec
\documentclass{svjour3}

\let\vec\origvec

\title{\solverref~-- A Solver for Linear Complementarity Quadratic Programs %
  \thanks{This research was supported by DFG via Research Unit FOR 2401 and project 424107692 and by the EU via ELO-X 953348.}
}%

\titlerunning{\solverref~-- A Solver for Linear Complementarity Quadratic Programs}
\authorrunning{Jonas Hall et al.}

\author{Jonas~Hall$^{1,2}$ \and Armin~\Nurkanovic$^{3}$ \and Florian~Messerer$^{3}$ \and Moritz~Diehl$^{1,3}$%
}%

\institute{%
    $^{1}$Department of Mathematics, University of Freiburg, Germany \\ %
    $^{2}$Division of Systems Engineering, Boston University, USA (present affiliation) \\ %
    $^{3}$Department of Microsystems Engineering (IMTEK), University of Freiburg, Germany
}%

\usepackage{printlen}
\uselengthunit{in}

\date{\today}

\usepackage{cite}
\usepackage{amssymb,amsfonts}
\usepackage{algorithmic}
\usepackage{mathtools}
\usepackage{graphicx}
\usepackage{textcomp}
\usepackage{xcolor}
\usepackage{hyperref}
\usepackage{pdflscape}
\usepackage[acronym,hyperfirst=false]{glossaries}
\usepackage{lmodern}
\usepackage[utf8]{inputenc}
\usepackage{graphicx}
\usepackage[super]{nth}
\usepackage{todonotes}
\usepackage{enumitem}
\usepackage{caption}
\usepackage{subcaption}
\usepackage[export]{adjustbox}
\usepackage{cases}
\graphicspath{{./figures}{./figures/benchmarks}{./LCQPTest/IVOCP/solutions/paper/}}

\newcommand{\secref}[1]{Section~\ref{#1}}
\newcommand{\lemref}[1]{Lemma~\ref{#1}}

\newcommand{\theoref}[1]{Theorem~\ref{#1}}
\renewcommand{\eqref}[1]{(\ref{#1})}
\newcommand{\figref}[1]{Figure~\ref{#1}}
\newcommand{\tabref}[1]{Table~\ref{#1}}

\newcommand{\algoref}[1]{Algorithm~\ref{#1}}

\DeclareMathOperator*{\minimize}{\mathrm{minimize}}

\usepackage{bm}

\usepackage{tikz}
\usetikzlibrary{positioning, arrows, calc}

\tikzset{
	>=stealth',
	help lines/.style={dashed, thick},
	axis/.style={<->},
	important line/.style={thick},
	connection/.style={thick, dotted},
}

\usepackage{pgfplots}
\pgfplotsset{compat=1.15}
\newlength\fheight
\newlength\fwidth
\renewcommand{\theenumi}{(\roman{enumi})}%

\usepackage{booktabs}
\newcommand{\ra}[1]{\renewcommand{\arraystretch}{#1}}

\usepackage[nocomma]{optidef}

\usepackage{multirow}

\newcommand{\R}{\mathbb{R}}

\newcommand{\N}{\mathbb{N}}

\newcommand{\sign}{\mathrm{sgn}}

\newcommand{\epsmach}{\varepsilon_{\textrm{mach}}}

\usepackage{upgreek}
\usepackage[ruled,vlined,linesnumbered]{algorithm2e}
\newcommand{\expnumber}[2]{{#1}\mathrm{e}{#2}}

\newcommand{\Nurkanovic}{Nurkanovi\'{c}}

\newcommand{\cpp}{C\texttt{++}}

\usepackage{bigdelim}

\newcommand{\xkj}{x_{kj}}

\newcommand{\pkj}{p_{kj}}
\newcommand{\alphakj}{\alpha_{kj}}

\newcommand{\thetakj}{\vartheta_{kj}}

\newcommand{\ellLi}{\ell_{L_i}}
\newcommand{\ellRi}{\ell_{R_i}}

\newcommand{\tomega}{\tilde{\Omega}}
\newcommand{\tA}{\tilde{A}}

\newcommand{\solverref}{LCQPow}

\newacronym[plural=LCQPs,firstplural=Linear Complementarity Quadratic Programs (LCQP)]{lcqp}{LCQP}{Linear Complementarity Quadratic Program}
\newacronym[plural=MPCCs,firstplural=Mathematical Programs with Complemenetarity Constraints (MPCC)]{mpcc}{MPCC}{Mathematical Program with Complemenetarity Constraints}
\newacronym[plural=QPs,firstplural=Quadratic Programs (QP)]{qp}{QP}{Quadratic Program}
\newacronym[plural=QCQPs,firstplural=Quadratically Constrained Quadratic Programs (QPQP)]{qcqp}{QCQP}{Quadratically Constrained Quadratic Program}
\newacronym[plural=NLPs,firstplural=Nonlinear Programs (NLP)]{nlp}{NLP}{Nonlinear Program}
\newacronym[plural=OCPs,firstplural=Optimal Control Problems (OCP)]{ocp}{OCP}{Optimal Control Problem}
\newacronym[plural=MIQPs,firstplural=Mixed Integer Quadratic Programs (MIQP)]{miqp}{MIQP}{Mixed Integer Quadratic Program}
\newacronym[plural=MINLPs,firstplural=Mixed Integer Non-Linear Programs (MINLP)]{minlp}{MINLP}{Mixed Integer Non-Linear Program}
\newacronym[plural=QPCCs,firstplural=Quadratic Programs with Complementarity Constraints (QPCC)]{qpcc}{QPCC}{Quadratic Program with Complementarity Constraints}
\newacronym[plural=QPLCCs,firstplural=Quadratic Programs with Linear Compementarity Constraints (QPLCC)]{qplcc}{QPLCC}{Quadratic Program with Linear Compementarity Constraints}
\newacronym[plural=FDIs,firstplural=Filippov Differential Inclusions (FDI)]{fdi}{FDI}{Filippov Differential Inclusion}

\newacronym{sqp}{SQP}{Sequential Quadratic Programming}
\newacronym{scp}{SCP}{Sequential Convex Programming}
\newacronym{licq}{LICQ}{Linear Independence Constraint Qualification}
\newacronym{mfcq}{MFCQ}{Mangasarian-Fromovitz Constraint Qualification}
\newacronym{kkt}{KKT}{Karush-Kuhn-Tucker}

\newcommand{\ac}[1]{\gls*{#1}}
\newcommand{\acp}[1]{\glspl*{#1}}

\glsunset{qp}

\newcommand{\useroption}{\texttt}
\newcommand{\maxRho}{\useroption{maxPenaltyParameter}}
\newcommand{\initialPenaltyParameter}{\useroption{initialPenaltyParameter}}
\newcommand{\penaltyUpdateFactor}{\useroption{penaltyUpdateFactor}}
\newcommand{\stationarityTolerance}{\useroption{stationarityTolerance}}
\newcommand{\complementarityTolerance}{\useroption{complementarityTolerance}}
\newcommand{\nDynamicPenalty}{\useroption{nDynamicPenalty}}
\newcommand{\etaDynamicPenalty}{\useroption{etaDynamicPenalty}}
\newcommand{\solveZeroPenaltyFirst}{\useroption{solveZeroPenaltyFirst}}
\newcommand{\printLevel}{\useroption{printLevel}}
\newcommand{\qpSolver}{\useroption{qpSolver}}
\newcommand{\maxIterations}{\useroption{maxIterations}}

\newcommand{\myqed}{\hfill~\qed}

\input{widebar}

\begin{document}

\maketitle

\begin{abstract}
    In this paper we introduce an open-source software package written in \cpp~for efficiently finding solutions to quadratic programming problems with linear complementarity constraints. These problems arise in a wide range of applications in engineering and economics, and they are challenging to solve due to their structural violation of standard constraint qualifications, and highly nonconvex, nonsmooth feasible sets. This work extends a previously presented algorithm based on a sequential convex programming approach applied to a standard penalty reformulation. We examine the behavior of local convergence and introduce new algorithmic features. Competitive performance profiles are presented in comparison to state-of-the-art solvers and solution variants in both existing and new benchmarks.
\end{abstract}

\newpage

\section{Introduction}
This paper presents the release 1.0 of LCQPow, which is a solver designed to efficiently solve \glspl{lcqp}. These problems can be expressed in the form
\begin{mini!}
	{x \in \R^n}{\frac{1}{2}x^\top Q x + g^\top x \label{min:LCQP:intro:obj}}
	{\label{min:LCQP:intro}}{}
	\addConstraint{b}{\leq A x, \label{min:LCQP:intro:A}}
	\addConstraint{0}{\leq Lx \perp Rx \geq 0, \label{min:LCQP:intro:complementarity}}
\end{mini!}
where $Q$ is assumed to be positive definite (a more detailed form matching the solver's API is later introduced in~\eqref{min:LCQP}). Note that the objective function together with constraint~\eqref{min:LCQP:intro:A} define a generic convex quadratic problem. The main difficulty of the above class arises through the nonlinear and nonconvex complementarity constraints. Their compact form~\eqref{min:LCQP:intro:complementarity} denotes the set of constraints
\begin{subnumcases}{0 \leq Lx \perp Rx \geq 0 \Longleftrightarrow}
	0 \leq Lx, \label{min:LCQP:nonneg:L} \\
	0 \leq Rx, \label{min:LCQP:nonneg:R} \\
	0 = x^\top L^\top R x. \label{min:LCQP:ortho}
\end{subnumcases}
We refer to the matrices $L$ and $R$ as the complementarity selector matrices. The pair of rows $L_i, R_i \in \R^{1 \times n}$ introduces the $i$th complementarity constraint, consisting of nonnegativity $0 \leq L_ix$, $0 \leq R_i x$ and orthogonality $x^\top L_i^\top R_i x = 0$. Note that the nonnegativity constraints~\eqref{min:LCQP:nonneg:L},~\eqref{min:LCQP:nonneg:R} make the orthogonality constraint~\eqref{min:LCQP:ortho} equivalent to complementarity satisfaction for each individual complementarity pair. The values of $L_i$ and $R_i$ define a weighted selection of the optimization variables to be present in the $i$th complementarity constraint. In the simplest case the selector matrices consist of unit vector rows, which imposes that at least one of the selected optimization variables vanishes for each complementarity pair. For example, this is the case in the two dimensional toy problem~\cite{scheel2000mathematical}
\begin{mini!}
	{x \in \R^2}{(x_1 - 1)^2 + (x_2 - 1)^2 \label{min:warm:up:objective}}
	{\label{min:warm:up}}{}
    \addConstraint{0}{\leq x_1 \perp x_2 \geq 0,}
\end{mini!}
which is illustrated in~\figref{fig:warm_up}.
\begin{figure}
    \begin{subfigure}[t]{0.5\linewidth}
        \includegraphics[width=0.9\linewidth,center]{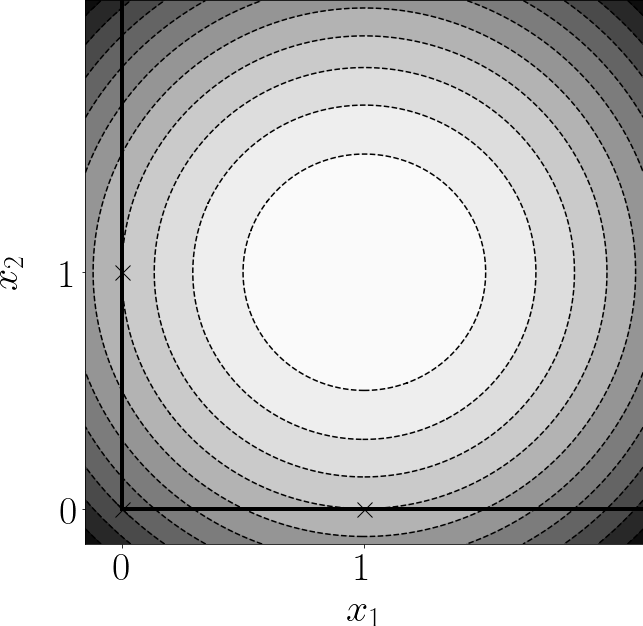}
    \end{subfigure}%
    \begin{subfigure}[t]{0.5\linewidth}
        \includegraphics[width=0.9\linewidth,center]{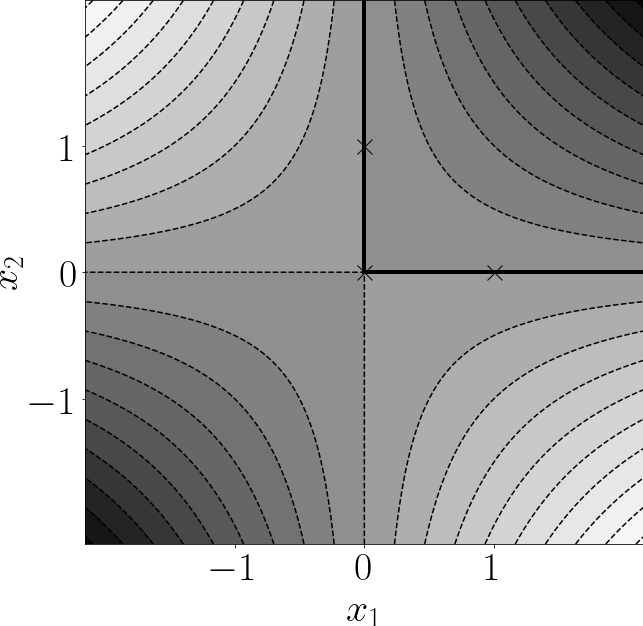}
    \end{subfigure}
    \caption{This illustration of the toy problem~\eqref{min:warm:up} shows the feasible set $\Omega$ (solid) together with the two strongly stationary points located at $(1, 0)$ and $(0,1)$, and the spurious solution in the origin. The level lines (dashed) represent the objective function on the left, and the penalty function $x_1 \cdot x_2$ (see~\eqref{min:LCQP:pen}) on the right.}
    \label{fig:warm_up}
\end{figure}

It is well known that problem~\eqref{min:LCQP} violates standard constraint qualifications such as the \ac{licq}, or even the weaker \ac{mfcq} at every feasible point~\cite{Ye1997}. Constraint regularity conditions are crucial assumptions for the concept of stationarity, e.g., for verifying \ac{kkt} points. This problem implies that standard approaches to solving~\acp{nlp} typically fail, creating a need for specialized methods.  Therefore, the theory of stationarity has been adapted and statements tailored to \acp{mpcc} have been developed~\cite{scheel2000mathematical}.

For recent advances in the field of \ac{mpcc}, including extensive lists of applications and methods, we refer to the surveys~\cite{ferris1997engineering,colson2007overview,Kim2020MPEC}. Within the subclass of convex quadratic objectives and linear constraints, there only exists a small amount of research ~\cite{Hall2021, bai2013convex,ralph2011c, chen2009class,deng2018globally}. These problems arise in a wide range of applications in engineering and economics, e.g., in optimal control problems of dynamical systems following discontinuous, but piecewise linear, dynamics~\cite{stewart1996numerical, stewart2010optimal}. Generally, systems of such dynamics are known as hybrid systems, and they have equivalently been modeled via mixed-logical dynamics~\cite{bemporad1999control, heemels2001equivalence}. Due to the combinatorial structure of the problem class it comes natural to investigate branch-and-bound methods in order to find global solutions~\cite{bai2013convex}. Recent advances proposed low-complexity methods for such systems~\cite{frick2019low}, motivated by the fact that mixed-integer solvers require high computational power and memory availability. Similarly, the intention of the solver presented here is to rapidly generate good local solutions with complementarity satisfaction up to machine precision.

The remainder of this paper is structured as follows. \secref{sec:background} provides background, existing methods and solvers for \acp{lcqp}. The algorithm is described in detail in \secref{sec:algorithm}, which builds upon the design originally presented in~\cite{Hall2021}. This includes an outer penalty loop, an inner \ac{scp} loop, an analytical globalization scheme, adaptive penalty updates and a heuristic for escaping saddle points. In \secref{sec:convergence} we state local convergence properties of the inner and outer loops, and provide statements for merit function descent at each inner loop iterate. The performance of the solver is benchmarked in \secref{sec:testing} against a variety of solvers and methods.

The contribution of this work primarily consists of the open-source software implementation written in \cpp. The code previously introduced provided the proof of concept for the underlying method, and is now transferred into a reliable, robust and efficient solver with extended flexibility and user options. This is supported by thorough benchmarks, which consist of the \ac{mpcc} benchmark~\cite{leyffer2000MacMPEC}, the benchmark discussed in~\cite{Hall2021}, and one benchmark created specifically for this paper. Additionally, the theoretical results are extended by a local convergence property, which states instant convergence on identification of a locally optimal active set.

\section{Background}\label{sec:background}

In this section we give a brief discussion of the background material, such as stationarity concepts and existing methods. We specifically lay our focus on the methods used for comparison within the numerical benchmarks in~\secref{sec:testing}. 

Let us begin by introducing a more generic form of the problem definition~\eqref{min:LCQP:intro}. Most importantly, this form enables users to pass arbitrary bounds on the complementarity variables. Whereas their upper bounds behave like simple linear constraints, their lower bounds have a special meaning: for each satisfied complementarity pair one of the lower bounds must be active. Hence, it is crucial to require their lower bounds to admit finite values. This form matches the solver's API and reads as 
\begin{mini!}
	{x \in \R^n}{\frac{1}{2}x^\top Q x + g^\top x \label{min:LCQP:obj}}
	{\label{min:LCQP}}{}
	\addConstraint{(Lx - \ell_L)^\top (Rx - \ell_R)}{ = 0, \label{min:LCQP:orthogonality:constraint}}
	\addConstraint{\ell_L \leq Lx}{\leq u_L, \label{min:LCQP:L}}
	\addConstraint{\ell_R \leq Rx}{\leq u_R, \label{min:LCQP:R}}
	\addConstraint{\ell_A \leq Ax}{\leq u_A, \label{min:LCQP:A}}
	\addConstraint{\ell_x \leq \phantom{A}x }{\leq u_x, \label{min:LCQP:x}}
\end{mini!}
where $0 \prec Q = Q^\top \in \R^{n \times n}$ , $g \in \R^n$, $L,R \in \R^{n_c \times n}$, $\ell_L, \ell_R, u_L, u_R \in \R^{n_c}$, $A \in \R^{n_A \times n}$, $\ell_A, u_A \in \R^{n_A}$, and $\ell_x, u_x \in \R^{n}$. We denote by $\Omega \subset \R^n$ the feasible set of~\eqref{min:LCQP}. 

Many \ac{qp} solvers exploit the special structure of the box constraints~\eqref{min:LCQP:x}, however, in view of the theoretical analysis and the higher level algorithm, these constraints can be seen as linear constraints. We thus assume throughout this paper that the box constraints~\eqref{min:LCQP:x} are passed via the linear constraints~\eqref{min:LCQP:A}. 

Further, let $\ell_{\tA}, \tA, u_{\tA}$ refer to the stacked combination of all linear constraints and their bounds~\eqref{min:LCQP:L}-\eqref{min:LCQP:x}. We refer to the resulting feasible set as the relaxed feasible set $\tomega \supseteq \Omega$ of~\eqref{min:LCQP}. Throughout this paper we assume that this relaxed feasible set satisfies \ac{licq} in every feasible point.

\subsection{Stationarity of LCQPs}
As mentioned, the considered problem class violates standard constraint qualifications required in order for the \ac{kkt} conditions to necessarily hold in solutions. We therefore review the adapted stationarity concept for complementarity constrained programs~\cite{scheel2000mathematical, guo2015solving, ralph2011c}. Let us first define the (in)active sets
\begin{subequations}
    \begin{align}
        \mathcal{A}^\mathrm{l}(x) &= \{ i \in \mathcal{J}_A \mid \ell_{A_i} = A_i x < u_{A_i} \}, \\
        \mathcal{A}^\mathrm{u}(x) &=   \{ i \in \mathcal{J}_A \mid \ell_{A_i} < A_i x = u_{A_i} \}, \\
        \mathcal{A}^\mathrm{e}(x) &= \{ i \in \mathcal{J}_A \mid \ell_{A_i} = A_i x = u_{A_i} \}, \\
        \mathcal{A}^\mathrm{f} (x) &=   \{ i \in \mathcal{J}_A \mid \ell_{A_i} < A_i x < u_{A_i} \},
    \end{align}
\end{subequations}
where $\mathcal{J}_A = \{1, \dots, n_A\}$. Analogously define the respective sets for the constraints $\ell_L \leq Lx \leq u_L$ and $\ell_R \leq Rx \leq u_R$ by $\mathcal{L}^\mathrm{l}, \mathcal{R}^\mathrm{l}$, etc. Further, let $\mathcal{W}^\mathrm{l}(x) = \mathcal{L}^\mathrm{l}(x) \cap \mathcal{R}^\mathrm{l}(x)$, $\bar{\mathcal{L}}^\mathrm{l} = \mathcal{L}^\mathrm{l} \setminus \mathcal{W}^\mathrm{l}$, and $\bar{\mathcal{R}}^\mathrm{l} = \mathcal{R}^\mathrm{l} \setminus \mathcal{W}^\mathrm{l}$. Note that any feasible point $x \in \Omega$ must satisfy $\mathcal{L}^\mathrm{l}(x) \cup \mathcal{R}^\mathrm{l}(x) = \{1,\dots,n_c\}$ due to the complementarity constraints~\eqref{min:LCQP:orthogonality:constraint}-\eqref{min:LCQP:R}.

The adapted stationarity concepts are very similar to the standard \ac{kkt} conditions of~\eqref{min:LCQP}, with the sole difference that the signs of the dual variables associated with constraints $\bar{\mathcal{L}}^\mathrm{l}(x)$ and $\bar{\mathcal{R}}^\mathrm{l}(x)$ are not required to be nonnegative. 

\begin{definition}
    A feasible point $x \in \Omega$ of \ac{lcqp}~\eqref{min:LCQP} is called strongly stationary, if there exist dual variables $y = (y_A, y_L, y_R) \in \R^{n_A} \times \R^{n_c} \times \R^{n_c}$ satisfying
    \begin{align}\label{eq:strong:stationarity}
        \begin{array}{rlrlrl}
        \multicolumn{5}{r}{Qx + g - A^\top y_A - L^\top y_L - R^\top y_R} & = 0, \\
        y_{A_i} & = 0,~i \in \mathcal{A}^\mathrm{f}(x), & \quad y_{A_i} & \geq 0,~i \in \mathcal{A}^\mathrm{l}(x), & \quad y_{A_i} & \leq 0,~i \in \mathcal{A}^\mathrm{u}(x), \\
        y_{L_i} & = 0,~i \in \mathcal{L}^\mathrm{f}(x), & \quad y_{L_i} & \geq 0,~i \in \mathcal{W}^\mathrm{l}(x), & \quad y_{L_i} & \leq 0,~i \in \mathcal{L}^\mathrm{u}(x), \\
        y_{R_i} & = 0,~i \in \mathcal{R}^\mathrm{f}(x), & \quad y_{R_i} & \geq 0,~i \in \mathcal{W}^\mathrm{l}(x), & \quad y_{R_i} & \leq 0,~i \in \mathcal{R}^\mathrm{u}(x).
        \end{array}
    \end{align}
\end{definition}
Other stationarity concepts further relax the sign requirement for the dual variables associated with $\mathcal{W}^\mathrm{l}(x)$~\cite{guo2015solving, ralph2011c}.

\subsection{Existing Methods}\label{sec:methods}
In this section we specify the methods utilized for comparison in \secref{sec:testing}. There are several techniques and solvers designed for \acp{mpcc}~\cite{murtagh1983minos,gill2005snopt,ralph2004some}. Additionally, we will consider a \ac{miqp} reformulation, such that any \ac{miqp} solver can be used. In this section we describe in detail three commonly used \ac{nlp} reformulations, each of which eliminates the orthogonality constraint~\eqref{min:LCQP:orthogonality:constraint} by introducing either regularized constraints or penalization~\cite{ralph2004some}.

\subsubsection{Penalty Reformulation}\label{sec:penalty:reformulations}
The first method replaces~\eqref{min:LCQP:orthogonality:constraint} with a penalty in the objective and reads as
\begin{mini!}
	{x \in \R^n}{\frac{1}{2}x^\top Q x + g^\top x + \rho \cdot (Lx - \ell_L)^\top (R x - \ell_R)\label{min:LCQP:pen:obj}}
	{\label{min:LCQP:pen}}{}
	\addConstraint{\ell_{\tA} \leq \tA x}{\leq  u_{\tA},}
\end{mini!}
where $\rho > 0$ is the respective penalty parameter. Note that the penalty term is always nonnegative due to~\eqref{min:LCQP:L} and~\eqref{min:LCQP:R}. The right plot in \figref{fig:warm_up} depicts the level lines of this penalty function for the toy problem~\eqref{min:warm:up}. The corresponding LCQP~\eqref{min:LCQP} is then approximated by solving~\eqref{min:LCQP:pen} either a single time with a large penalty value, or sequentially with an exponentially increasing penalty value~\cite[Section 4]{ferris1999Solution}. Ralph and Wright proved that this is an exact penalty reformulation for large enough, but finite, values of $\rho$~\cite[Section 5]{ralph2004some}:
\begin{theorem}\label{theo:penalty:convergence}
    Let~\eqref{min:LCQP:pen} satisfy \ac{licq} at $x^\ast \in \R^n$. Then the following statements hold:
        \begin{enumerate}
            \item\label{theo:penalty:convergence:1} If $(x^\ast, y_A^\ast, y_L^\ast, y_R^\ast)$ is a strongly stationary point of the \ac{lcqp}~\eqref{min:LCQP}, then there exist dual variables $(\bar{y}_A, \bar{y}_L, \bar{y}_R)$ such that $x^\ast$ is a \ac{kkt} point of~\eqref{min:LCQP:pen} for any $\rho$ satisfying
            \begin{equation}\label{eq:penalty:bound}
                \rho \geq 1 + \max \left\{ 0, \max_{i \in \bar{\mathcal{L}}^\mathrm{l}(x^\ast)}\left\{ \frac{-y_{L_i}^\ast}{R_i x^\ast - \ellRi} \right\}, \max_{i \in \bar{\mathcal{R}}^\mathrm{l}(x^\ast)}\left\{ \frac{-y_{R_i}^\ast}{L_i x^\ast - \ellLi} \right\} \right\}.
            \end{equation}
            The dual variables are given by
            \begin{subequations}\label{eq:dual:switch}
                \begin{align}
                    \bar{y}_A &= y_A^\ast, \\
                    \bar{y}_{L_i} &= y_{L_i}^\ast,                                &&\mathrm{for~} i \notin \mathcal{L}^\mathrm{l}(x^\ast),\\
                    \bar{y}_{R_i} &= y_{R_i}^\ast,                                &&\mathrm{for~} i \notin \mathcal{R}^\mathrm{l}(x^\ast),\\
                    \bar{y}_{L_i} &= y_{L_i}^\ast + \rho (R_i x^\ast - \ellRi), &&\mathrm{for~} i \in \mathcal{L}^\mathrm{l}(x^\ast),\\
                    \bar{y}_{R_i} &= y_{R_i}^\ast + \rho (L_i x^\ast - \ellLi), &&\mathrm{for~} i \in \mathcal{R}^\mathrm{l}(x^\ast),
                \end{align}
            \end{subequations}

            \item\label{theo:penalty:convergence:2} If $(x^\ast, \bar{y}_A, \bar{y}_L, \bar{y}_R)$ is a \ac{kkt} point of~\eqref{min:LCQP:pen} and $(Lx^\ast - \ell_L)^\top (R x^\ast - \ell_R) = 0$, then $(x^\ast, y_A^\ast, y_L^\ast, y_R^\ast)$ is a strongly stationary point of the \ac{lcqp}~\eqref{min:LCQP}, where the dual variables are obtained via~\eqref{eq:dual:switch}.
        \end{enumerate}
\end{theorem}
The algorithm implemented in the presented software package is based on this penalty reformulation, and we thus focus on this specific penalty function, though other choices are conceivable as well~\cite{fischer1995ncp,abdallah2019solving,chen2000penalized}.

\subsubsection{Constraint Regularization Reformulations}
The remaining two \ac{nlp} methods both replace~\eqref{min:LCQP:orthogonality:constraint} by constraint regularization strategies, each of which uses a parameter $\sigma > 0$. These methods approximate~\eqref{min:LCQP} by
\begin{mini!}
	{x \in \R^n}{\frac{1}{2}x^\top Q x + g^\top x}
	{\label{min:LCQP:nlp:smooth}}{}
	\addConstraint{(Lx - \ell_{L})^\top(R x - \ell_{R})}{= \sigma,}
	\addConstraint{\ell_{\tA} \leq \tA x}{\leq  u_{\tA},}
\end{mini!}
which we call the smoothed reformulation of~\eqref{min:LCQP}, and by
\begin{mini!}
	{x \in \R^n}{\frac{1}{2}x^\top Q x + g^\top x}
	{\label{min:LCQP:nlp:reg}}{}
	\addConstraint{(Lx - \ell_L)^\top (R x - \ell_R)}{\leq \sigma,}
	\addConstraint{\ell_{\tA} \leq \tA x}{\leq  u_{\tA},}
\end{mini!}
which we call the relaxation of~\eqref{min:LCQP}. In contrast to the penalty parameter $\rho$, these methods approximate~\eqref{min:LCQP} for $\sigma \to 0$. As for the penalty reformulation, they naturally lend themselves to being used in a sequential scheme with exponentially decaying choices for $\sigma$. 

\subsubsection{\ac{miqp} Reformulation}
Finally, we address a direct reformulation into a mixed-integer quadratic program. This is straightforward under the existence of finite upper bounds on the complementarity variables, i.e., given $u_L, u_R < \infty$. In that case we may simply introduce two binary variables $(z_i^L, z_i^R)$ for each complementarity constraint $i \in \{1,2,\dots,n_c\}$. We can then enforce complementarity via the set of constraints 
\begin{subequations}
    \begin{align}
        \ell_{L_i} &\leq L_i x, \\
        \ell_{R_i} &\leq R_i x, \\
        L_i x &\leq u_{L_i} z_i^L + (1-z_i^L)\ell_{L_i}, \\
        R_i x &\leq u_{R_i} z_i^R + (1-z_i^R)\ell_{R_i}, \\
        1 &\geq z_i^L + z_i^R.
    \end{align}
\end{subequations}
Thus, for $z_i^L = 1$ we simply regain the upper bound. On the other hand, $z_i^L = 0$ fixes the complementarity variable to its lower bound. In total the reformulation reads as
\begin{mini!}
	{\substack{\textstyle{x \in \R^n}, \\ \textstyle{z^L,z^R \in B^{n_c}}}}{\frac{1}{2}x^\top Q x + g^\top x}
	{\label{eq:MIQP}}{}
	\addConstraint{\ell_{\tA}}{\leq \tA x \leq  u_{\tA},}
    \addConstraint{L_i x}{\leq u_{L_i} z_i^L + (1-z_i^L)\ell_{L_i}, }{\quad 1 \leq i \leq n_c,}
    \addConstraint{R_i x}{\leq u_{R_i} z_i^R + (1-z_i^R)\ell_{R_i}, }{\quad 1 \leq i \leq n_c,}
    \addConstraint{1}{\geq z^L + z^R,}
\end{mini!}
where $B = \{0,1\}$. If upper bounds on the complementarity variables are not defined, then we rely on a big-$M$ reformulation, by setting the upper bounds to a large $M \gg 0$.

\section{Algorithm}\label{sec:algorithm}
In this section we describe in detail the initial algorithmic development~\cite{Hall2021} and provide several extensions. We begin by considering the penalty reformulation~\eqref{min:LCQP:pen} together with an exponential penalty update rule similar to the one described by Ferris et al. in~\cite{ferris1999Solution}. This technique leads to a sequence of nonconvex quadratic programming problems (which we denote by the outer loop below). Solutions to each of these problems are then found via a \ac{scp} method (which provides the inner loop). The key steps of the algorithm are captured in the pseudocode~\algoref{algo:LCQP} at the end of this section. In this section we introduce many parameters, the matching API names of which we mention as (\useroption{parameterName}) whenever newly introduced. \tabref{tab:user:options} summarizes their default values and feasible range. Finally, we remind the reader that, in contrast to~\cite{Hall2021}, the complementarity constraints are generalized by allowing generic lower and upper bounds, which creates some subtle differences.

\subsection{Penalty Homotopy}\label{sec:penalty:homotopy}
Motivated by \theoref{theo:penalty:convergence}, we desire a solution of the penalty reformulation~\eqref{min:LCQP:pen} for a penalty large enough to satisfy the complementarity constraints. However, the required penalty value is a priori unknown. One could consider simply solving~\eqref{min:LCQP:pen} for a very large penalty in the hope of instantly satisfying complementarity. However, these penalty reformulations often become ill-conditioned for large penalty parameters, i.e., the largest absolute eigenvalue of $Q + \rho C$ significantly dominates the smallest absolute eigenvalue. Furthermore, a homotopy often avoids convergence to strongly suboptimal solutions, as for example shown for \acp{ocp} with discontinuous dynamics~\cite{nurkanovic2020limits}. On the other hand, solving the penalized subproblem for a very small penalty parameter leads to a solution close to the global minimum $\tilde{x}^\ast$ of the objective function \eqref{min:LCQP:obj} with respect to the relaxed feasible set $\tomega$. By gradually increasing the penalty parameter we hope to find a solution path from the relaxed minimizer $\tilde{x}^\ast$ to a strongly stationary point $x^\ast \in \Omega$ approximated from within the relaxed feasible set $\tilde{\Omega}$. Yet, this is only a heuristic and there is no guarantee of finding the global minimizer, or even any minimizer, as the original NLP~\eqref{min:LCQP} is nonconvex.

Let us now describe the homotopy. For a given penalty parameter $\rho_k > 0$ the respective penalty reformulation~\eqref{min:LCQP:pen} is solved as described in the next section. Subsequently, the penalty parameter is updated as $\rho_{k+1} = \beta \rho_k$ with a fixed factor $\beta > 1$. This method also requires the choice of an initial penalty parameter $\rho_0 > 0$ ({\initialPenaltyParameter}), which is typically chosen rather small. The factor $\beta$ (\penaltyUpdateFactor) represents the base of the exponential growth, and one could alternatively write $\rho_k = \beta^k \rho_0$. This procedure is repeated until complementarity is satisfied, or the penalty parameter exceeds its limit ({\maxRho}), in which case the convergence is assumed to have failed. 

Before proceeding with the inner loop, let us refine the penalty formulation~\eqref{min:LCQP:pen}. We introduce the penalty function
\begin{equation}\label{eq:phi}
    \varphi(x) = (Lx - \ell_L)^\top (Rx - \ell_R) = \frac{1}{2} x^\top C x + g_\varphi^\top x + \ell_L^\top \ell_R,
\end{equation}
where $\frac{1}{2}C = \frac{1}{2}\left(L^\top R + R^\top L\right)$ is the symmetrization of the product $L^\top R$, and $g_\varphi = -(R^\top \ell_L + L^\top \ell_R)$ is the linear component of $\varphi$. We remark here that $C$ is typically indefinite. If it does not contain negative eigenvalues, then the penalty reformulation is convex and its unique solution satisfies complementarity. We then combine the linear components of the objective function~\eqref{min:LCQP:pen:obj} by defining $g_k = g + \rho_kg_\varphi$. Finally, we obtain the following optimization problem, that is equivalent to~\eqref{min:LCQP:pen}
\begin{mini!}
	{x \in \R^n}{\frac{1}{2}x^\top Q x + g_k^\top x + \frac{\rho_k}{2}~x^\top C x \label{min:LCQP:pen:outer:obj}}
	{\label{min:LCQP:pen:outer}}{}
	\addConstraint{\ell_{\tA} \leq \tA x}{\leq  u_{\tA},}
\end{mini!}
and call the sequence of solving these problems for increasing $\rho_k$ the outer loop.

\subsection{Sequential Convex Programming}
Each outer loop problem is solved using \ac{scp}~\cite{messerer2021survey}, resulting in an inner loop. Let $k$ and $j$ denote the outer and inner loop indices, respectively. We denote by $\xkj$ the most recent inner \ac{scp} iterate. The sequence is initialized with an initial guess $x_{00}$, or alternatively the global minimizer of the relaxed problem (see~\secref{sec:initialization}).

The penalty function, which is the only nonconvex component of~\eqref{min:LCQP:pen:outer}, is approximated at $\xkj$ using its first-order Taylor expansion
\begin{equation*}
    \begin{split}
        \varphi(x) &\approx \varphi(\xkj) + (x - \xkj)^\top \nabla \varphi(\xkj) \\
        &= \left( \varphi(\xkj) - \xkj^\top (C \xkj + g_\varphi) \right) + x^\top (C \xkj + g_\varphi).
    \end{split}
\end{equation*}
Note that $x^\top (C \xkj + g_\varphi)$ is the only term dependent on $x$. Since the constant terms do not affect the optimizer, we omit them from here on. Replacing the penalty function by this term yields the convex inner loop subproblem
\begin{mini!}
	{x \in \R^n}{\frac{1}{2}x^\top Q x + \left(g_k + \rho_k C \xkj\right)^\top x  \label{min:LCQP:pen:inner:obj}}
	{\label{min:LCQP:pen:inner}}{}
	\addConstraint{\ell_{\tA} \leq \tA x}{\leq  u_{\tA}.}
\end{mini!}
%Note here that $g_k$, as defined in \secref{sec:penalty:homotopy}, is constant for the entirety of each inner loop. 
We denote the unique minimizer of the inner loop subproblem by $\xkj^\ast$ and the corresponding step direction by $\pkj = \xkj^\ast - \xkj$. Given this inner solution an optimal step length $\alphakj$ is obtained from a globalization scheme described in \secref{subsec:globalization}. Finally, the step update $x_{k,j+1} = \xkj + \alphakj \pkj$ is performed. The inner loop is terminated once a KKT point of the respective outer loop problem \eqref{min:LCQP:pen:outer} is found.

There are two reasons why it can be attractive to replace the full penalty function by its linear approximation. First, convex subproblems are obtained at the cost of the additional inner loop. Noting that the eigenvalues of $C$ will dominate over those of $Q$ for large penalty parameters, we find that the penalty formulation~\eqref{min:LCQP:pen:outer} becomes more and more indefinite as the penalty parameter grows. On the other hand, convexity of the inner loop subproblem~\eqref{min:LCQP:pen:inner} is always ensured, as the Hessian matrix is given by $Q$ for every subproblem. This also induces the second advantage: the Hessian and constraint matrices remain constant over both the inner and outer loop iterates. Consequently, the KKT matrix factorization can be reused, and each subproblem can be solved efficiently, e.g., by making use of the warm-starting techniques employed in QP solvers such as qpOASES~\cite{Ferreau2008} or OSQP~\cite{stellato2018embedded}. With the computation of factorizations being a significant expense, this advantage can outweigh the cost of inner loop iterations, as demonstrated in \secref{sec:testing}. We will also see that the \ac{scp} loop terminates finitely near an exact solution if $\rho_k$ is large enough (see~\theoref{theo:single:solve}).

\subsection{Optimal Step Length Globalization} \label{subsec:globalization}
Consider the merit function
\begin{equation}
    \psi_k(x) = \frac{1}{2} x^\top (Q + \rho_k C) x + g_k^\top x,
\end{equation}
which coincides with the outer loop objective function~\eqref{min:LCQP:pen:outer:obj}. On the other hand, the inner loop objective function~\eqref{min:LCQP:pen:inner:obj} provides the strictly convex quadratic model
\begin{equation}\label{eq:quadratic:model}
    \thetakj(x) = \frac{1}{2} x^\top Q x + (g_k + \rho_k C \xkj)^\top x.
\end{equation}
As discussed in~\cite{Hall2021}, the step length formula is obtained by minimizing the merit function along the step $\pkj$, i.e., by solving
\begin{equation}\label{min:merit:function}
    \minimize_{\alphakj \in [0,1]}\quad \psi_k(\xkj + \alphakj \pkj).
\end{equation}
This yields a scalar QP and its analytical solution is given by
\begin{equation}\label{eq:step:length:formula}
    \alphakj^\ast =
    \begin{cases}
        \dfrac{-\nabla \psi_k(\xkj)^\top \pkj}{\pkj^\top Q  \pkj + \pkj^\top  \rho_k C \pkj}, & \mathrm{if}~\pkj^\top C \pkj > 0, \\
        1, & \mathrm{else}.
    \end{cases}
\end{equation}

\begin{remark}
    The step length $\alphakj^\ast$ is strictly positive if $\xkj$ is not already a KKT point of the outer loop problem. We will discuss this in more detail in \secref{sec:convergence} by proving strict merit function descent in direction $\pkj$, i.e., $\nabla \psi_k(\xkj)^\top \pkj < 0$.
\end{remark}

\begin{remark}
    The formula presented in~\cite{Hall2021} was derived for a slightly less generic form, as generic bounds on the complementarity pairs were not permitted. However, this only changes the linear component $g_k$, which remains constant for each inner loop.
\end{remark}

\subsection{Dynamic Penalty Updates}\label{sec:dynamic:penalty}
It is possible that the inner loop requires many iterates until a satisfactory level for convergence is reached, while the progress of merit function descent might stagnate. In the context of interior point methods for \acp{mpcc}, Leyffer et al. have shown that it can be advantageous to terminate the inner loop prematurely and update the penalty parameter dynamically~\cite[Section 5]{leyffer2006interior}. This dynamic update is triggered whenever an inner loop iterate satisfies
\begin{equation} \label{eq:dynamic:penalty:check}
    \varphi(\xkj) > \varepsilon_\varphi \quad \mathrm{and} \quad \varphi(\xkj) > \eta \max\{ \varphi(x_{k,j-1}), \dots, \varphi(x_{k,j-n}) \},
\end{equation}
where $\varepsilon_\varphi > 0$ describes the numerical tolerance for the complementarity violation ({\complementarityTolerance}).
This method assures that, until complementarity is satisfied, each iterate reduces one of the previous $n$ complementarity violations by at least a factor of $\eta$. We embedded this strategy into our solver with the options {\nDynamicPenalty} and {\etaDynamicPenalty}. This strategy can be switched off by setting ${\nDynamicPenalty} = 0$.

\subsection{Initialization Strategy}\label{sec:initialization}
The initial guess is often a crucial factor for finding good local solutions of nonlinear programs. Thus it is desirable to initialize solvers in the basis of attraction of a good -- ideally global -- solution. The presented solver contains the option {\solveZeroPenaltyFirst}, a flag indicating whether the sequence should be initialized by solving \eqref{min:LCQP:pen:outer} with $\rho = 0$ (recall that this problem is convex, and its solution is the global minimizer of the objective function over the relaxed set $\tomega$). This canonical choice makes passing an initial guess optional.

However, this method becomes disadvantageous if proximity to a good solution is known. In this case the flag should be disabled and the solver should be initialized with a large penalty parameter in order to prevent the solver from leaving the area of attraction of the local solution. This is especially the case if the active set of the global solution has been identified (see \theoref{theo:single:solve}).

\begin{figure}
    \begin{subfigure}[b]{.5\linewidth}
        \includegraphics[width=0.95\linewidth,left]{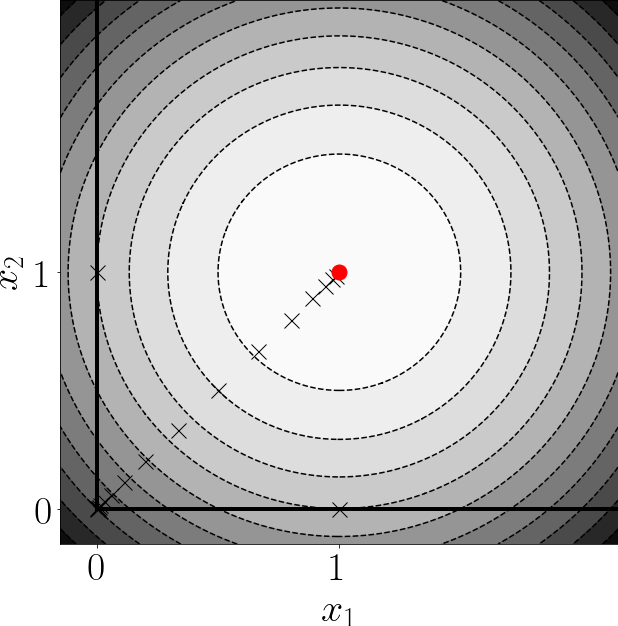}
        \caption{$\rho = 0$}
        \label{fig:warm_up:rho:0}
    \end{subfigure}%
    \begin{subfigure}[b]{.5\linewidth}
        \includegraphics[width=0.95\linewidth,right]{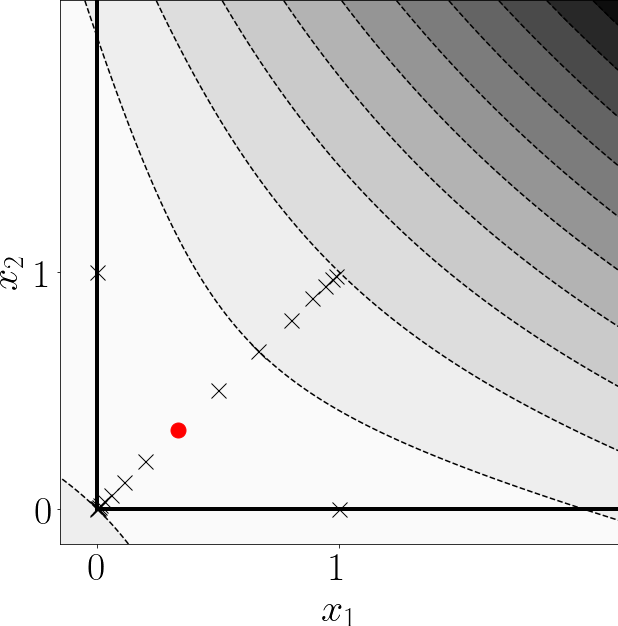}
        \caption{$\rho = 4$}
        \label{fig:warm_up:rho:4}
    \end{subfigure}
    \caption{The penalty reformulation of the LCQP~\eqref{min:warm:up} has a stable trajectory of saddle points (minimizers for $\rho < 2$) along the points $\bigl(2/(2 + \rho), 2/(2 + \rho)\bigr)$ marked with the black crosses. The red dot indicates the saddle point of the penalized objective of which the level lines are indicated.}
    \label{fig:saddle:points}
\end{figure}

\subsection{Gradient Perturbation}
Some problems might have a stable trajectory of minimizers or saddle points towards undesirable solutions. We demonstrate this issue using the toy problem~\eqref{min:warm:up}. In this case the standard strategy would initialize at $(1,1)$ and follow the saddle point trajectory into the origin and terminate at this locally maximal solution (see \figref{fig:saddle:points}). With a small zero mean random perturbation of the gradient $g_k$ at each step, we move the iterates away from the saddle points until the error is large enough for the QP solver to detect descent towards one of the strongly stationary points $(1, 0)$ or $(0, 1)$. Alternatively, one could consider applying the perturbation to the step directly. However, this would require additional safety-checks in order to dodge infeasibility. Perturbing the gradient is safe in this aspect as it only alters the objective function.

\subsection{Termination Criterion}\label{sec:termination}
We terminate the algorithm under three different scenarios: either a solution is found, or the penalty parameter is too large, or the maximum number of iterations are exceeded. The termination criterion for a solution consists of \ac{kkt} point verification of an iterate $\xkj$ for the penalty formulation~\eqref{min:LCQP:pen} together with sufficient complementarity satisfaction. The tolerances of both conditions can be adapted via the options {\stationarityTolerance} and {\complementarityTolerance}, respectively. Note that any feasible iterate $\xkj$ of~\eqref{min:LCQP:pen:inner} satisfies all constraints of~\eqref{min:LCQP}, except for the orthogonality constraint~\eqref{min:LCQP:orthogonality:constraint}. Thus it is sufficient to check
\begin{equation*}
    \begin{split}
        \| (Q + \rho_k C)\xkj + g_k - A^\top y_A - L^\top y_L - R^\top y_R \|_\infty &\leq {\stationarityTolerance}, \\
        \varphi(\xkj) &\leq {\complementarityTolerance}.
    \end{split}
\end{equation*}
The remaining feasibility conditions are assumed to be transferred through the QP solver. Since the underlying QP solvers use inherently different termination criteria, it is difficult to provide bounds on how to choose the termination tolerances precisely. If the termination conditions for \solverref~are chosen too small, then precision errors from the utilized QP solver may interfere with convergence. In this case, the precision for \solverref~should be decreased, i.e., the tolerance increased (or vice versa the QP solver precision should be increased). We remark here that, on successful convergence, the dual variables obtained from the penalty reformulation are translated into dual variables of the original \ac{lcqp}~\eqref{min:LCQP:obj} using~\eqref{eq:dual:switch}.

\subsection{QP Solvers}\label{sec:qp:solvers}
Through the user option {\qpSolver} it is possible to switch between the three modes $0,1,$ and $2$. The mode $0$ refers to qpOASES~\cite{Ferreau2014} in dense mode, $1$ refers to qpOASES in sparse mode, and $2$ refers to OSQP~\cite{osqp} (in sparse mode).

The performance of mode $1$ will depend on how qpOASES is compiled. If Matlab is installed on the machine one can pass the CMake option \useroption{-DQPOASES\_SCHUR=ON} to compile qpOASES with the Schur Complement method~\cite[Chapter 8]{janka2015schur}, which uses the sparse linear solver MA57~\cite{ma57}.

\subsection{Print Level}\label{sec:print:level}
The solver prints some information about the iterates to the command line, the amount of which can be controlled via the user option {\printLevel}. If no iterate output is desired then $0$ can be passed. Mode $1$ will print only one iterate of each inner loop. Mode $2$ will print every iterate.

\begin{algorithm}
    \vspace{0.5em}
    \KwIn{$\rho>0,~\beta>1,~\varepsilon_\mathrm{tol}>0$}
    \KwOut{Stationary point $(x_k, y_k)$ of LCQP~\eqref{min:LCQP}}\vspace{0.5em}

    \textcolor{gray}{\# Create QP solver and factorize KKT matrix} \\
    $\texttt{qp}(Q, g, \tilde{A}, \ell_x, u_x, \ell_{\tilde{A}}, u_{\tilde{A}})$; \label{line:factorize}\\[0.5em]

    \textcolor{gray}{\# Initialize solver with zero penalty QP} \\
    \texttt{$(x_k, y_k) = \texttt{qp.solve}()$}; \label{line:init}\\[0.5em]

    \textcolor{gray}{\# Outer loop (penalty update loop)} \\
    \While{\textup{true}}{ \label{line:while:outer} \vspace{0.5em}

        \textcolor{gray}{\# Update outer loop linear term}\\
        \texttt{$g_k = g - \rho(R^\top \ell_L + L^\top \ell_R)$}; \\[0.5em]

        \textcolor{gray}{\# Inner loop (approximate penalty function)} \\
        \While{$\textnormal{\texttt{stationarity}}(x_k, y_k, \rho) > \varepsilon_\mathrm{tol}$}{ \label{line:while:inner} \vspace{0.5em}
            \textcolor{gray}{\# Update objective's linear component}\\
            \texttt{$\texttt{qp.update\_g}(g_k + \rho C x_k)$}; \label{line:update:g} \\[0.5em]

            \textcolor{gray}{\# Step computation and dual variable update (solve~\eqref{min:LCQP:pen:inner})}\\
            \texttt{$(x_n, y_k) = \texttt{qp.solve()}$}; \label{line:hotstart} \\[0.5em]

            \textcolor{gray}{\# Get optimal step length according to~\eqref{eq:step:length:formula}}\\
            \texttt{$\alpha = \texttt{get\_step\_length}(x_k, x_n, \rho)$};\label{line:optimalStepLength} \\[0.5em]

            \textcolor{gray}{\# Update primal variables}\\
            \texttt{$x_k = x_k + \alpha(x_n - x_k)$}; \\[0.5em]

            \textcolor{gray}{\# Perform dynamic penalty update}\\
            \If{ \textup{\texttt{Condition~\eqref{eq:dynamic:penalty:check} holds}}}{ \label{line:dynamic:update:check}
            \texttt{break}\;
            }
        }\vspace{0.5em}

        \textcolor{gray}{\# Terminate if stationarity and complementarity are satisfied}\\
        \If{ $\varphi(x_k) < \varepsilon_\mathrm{tol}$ }{ \label{line:comp:check}
        \texttt{return $(x_k, y_k)$}\;
        }\vspace{0.5em}

        \textcolor{gray}{\# Increase penalty parameter}\\
        \texttt{$\rho = \beta\cdot\rho$}; \\[0.5em]
    }
    \caption{Pseudocode of the solver's main loop.}
    \label{algo:LCQP}
\end{algorithm}

\subsection{Software}
The open-source software package written in \cpp~is available through the GitHub repository
\[ \href{https://github.com/hallfjonas/LCQPow}{\texttt{https://github.com/hallfjonas/LCQPow}} \]
Version \texttt{v0.1.0} was used for this paper. This repository contains three submodules, which have to be initialized after cloning the repository. Those modules are the QP solvers qpOASES~\cite{Ferreau2014} and OSQP~\cite{osqp}, and the unit test framework GoogleTest~\cite{gtest}. The solver can be called either directly through \cpp~or through its Matlab interface. The user options with default values and feasible range are listed in \tabref{tab:user:options}.
\begin{table}
    \centering
    \ra{1.3}
    \caption{User options with their default values and feasible range}
    \begin{tabular}{@{}llll@{}}\toprule
        ~Parameter Name &
        Default Value &
        Feasible Values &
        Section \\ \midrule
        ~{\stationarityTolerance} & $\expnumber{1}{+6}\cdot\epsmach$ & $\R_{>0}$ & \secref{sec:termination} \\
        ~{\complementarityTolerance} & $\expnumber{1}{+3}\cdot\epsmach$ & $\R_{>0}$ & \secref{sec:termination} \\
        ~{\initialPenaltyParameter} & $\expnumber{1}{-2}$ & $\R_{>0}$ & \secref{sec:penalty:homotopy} \\
        ~\penaltyUpdateFactor & 2 & $\R_{>1}$ & \secref{sec:penalty:homotopy} \\
        ~{\solveZeroPenaltyFirst} & 1 & $\{0, 1\}$ & \secref{sec:initialization} \\
        ~{\maxIterations} & $\expnumber{1}{+3}$ & $\N$ & \secref{sec:termination} \\
        ~{\maxRho} & $\expnumber{1}{+4}$ & $\R_{>0}$ & \secref{sec:termination} \\
        ~{\printLevel} & 2 & $\{0, 1, 2\}$ & \secref{sec:print:level} \\
        ~{\qpSolver} & 0 & $\{0, 1, 2\}$ & \secref{sec:qp:solvers} \\
        ~{\nDynamicPenalty} & 3 & $\N$ & \secref{sec:dynamic:penalty} \\
        ~{\etaDynamicPenalty} & 0.9 & $(0,1)$ & \secref{sec:dynamic:penalty} \\
        \bottomrule
    \end{tabular}
    \label{tab:user:options}
\end{table}

\section{Convergence Analysis} \label{sec:convergence}
We now draw our attention to the local convergence behavior of the above introduced algorithm. We begin by revisiting in more detail the properties introduced in~\cite{Hall2021}: a relationship between the minimizers of the inner loop problem with the KKT points of the outer loop problem (\lemref{lem:inner:outer:stationarity}); and strict merit function descent in each inner loop iterate until convergence is reached (\theoref{theo:direction:of:descent}). This section is concluded with the local convergence statement~\theoref{theo:single:solve}, which shows that the algorithm converges in one step once the active sets of the complementarity pairs of the current iterate coincide with those of a strongly stationary point.
\begin{lemma}\label{lem:inner:outer:stationarity}
    Let $(\xkj, \rho_k)$ be a feasible iterate of~\eqref{min:LCQP:pen:inner}. Then the respective inner loop minimizer $\xkj^\ast$ is equal to $\xkj$ iff $\xkj$ is a \ac{kkt} point of the outer loop problem~\eqref{min:LCQP:pen:outer}.
\end{lemma}
\begin{proof}
    For a proof of this standard result we refer to~\cite[Lemma 4.1]{messerer2021survey}.
\end{proof}

\begin{theorem}\label{theo:direction:of:descent}
    Given $\xkj \in \tomega$ with inner loop step $\pkj = \xkj^\ast - \xkj$, the merit function at $\xkj$ is nonincreasing in direction $\pkj$, i.e.,
    \begin{equation}\label{eq:descent}
        \nabla \psi_k(\xkj)^\top (\xkj^\ast - \xkj) \leq 0.
    \end{equation}

    Furthermore, if $\xkj$ is not a stationary point of~\eqref{min:LCQP:pen:outer} (with respect to $\rho_k$), then
    \begin{equation}\label{eq:strict:descent}
        \nabla \psi_k(\xkj)^\top (\xkj^\ast - \xkj) < 0.
    \end{equation}
\end{theorem}
\begin{proof}
    Since $\xkj^\ast$ is the global minimum of the inner loop optimization problem, the following relation holds
    \begin{equation}\label{ineq:minimizer}
        \thetakj(\xkj^\ast) \leq \thetakj(x),
    \end{equation}
    where $x$ is any feasible point of~\eqref{min:LCQP:pen:outer} and $\thetakj$ is as defined in~\eqref{eq:quadratic:model}. Since $\thetakj$ is convex and differentiable it holds for any $a, b \in \R^n$ that
    \begin{equation}
        \nabla \thetakj(a)^\top (b-a) \leq \thetakj(b) - \thetakj(a).
    \end{equation}
    This property provides descent for the quadratic model
    \begin{equation}\label{eq:quadratic:model:descent}
        \nabla \thetakj(\xkj)^\top \pkj \leq \thetakj(\xkj^\ast) - \thetakj(\xkj) \leq 0.
    \end{equation}
    Note that this inequality becomes strict if $\xkj \neq \xkj^\ast$, since~\eqref{ineq:minimizer} becomes strict. Further, we have
    \begin{subequations}\label{eq:model:and:merit:descent:agree}
        \begin{align}
            \nabla \psi(\xkj, \rho_k)^\top \pkj &= (Q \xkj + \rho_k C \xkj + g)^\top \pkj \\
            &= \nabla \thetakj(\xkj)^\top \pkj,
        \end{align}
    \end{subequations}
    which shows that the directional derivatives of the merit function and quadratic model at $\xkj$ towards $\pkj$ agree. Inequality~\eqref{eq:descent} immediately follows.

    Assume that $\xkj$ is not outer loop stationary. Then \lemref{lem:inner:outer:stationarity} yields $\xkj^\ast \neq \xkj$. As remarked before, the inequality~\eqref{eq:quadratic:model:descent} becomes strict, and plugging in~\eqref{eq:model:and:merit:descent:agree} concludes the claim. \myqed
\end{proof}

\begin{theorem}\label{theo:single:solve}
    Let $(x^\ast, y^\ast)$ be a strongly stationary point of~\eqref{min:LCQP} and let $\xkj \in \Omega$ be an iterate of the algorithm such that $\mathcal{L}^\mathrm{l}(\xkj) = \mathcal{L}^\mathrm{l}(x^\ast)$ and $\mathcal{R}^\mathrm{l}(\xkj) = \mathcal{R}^\mathrm{l}(x^\ast)$, i.e., the respective active complementarity sets agree. Further, let $\rho_k$ be sufficiently large. Then the next inner loop iterate will coincide with the strongly stationary point, i.e., $x_{k,j+1} = x^\ast$, and the algorithm will return the strongly stationary point $(x^\ast, y^\ast)$.
\end{theorem}
\begin{proof}
    Similar to \theoref{theo:penalty:convergence} we define
    \begin{subequations}\label{eq:duals:proof:single:solve}
        \begin{align}
            \bar{y}_A &= y_A^\ast, \\
            \bar{y}_{L_i} &= y_{L_i}^\ast,                             &&\mathrm{for~} i \notin \mathcal{L}^\mathrm{l}(\xkj),\\
            \bar{y}_{R_i} &= y_{R_i}^\ast,                             &&\mathrm{for~} i \notin \mathcal{R}^\mathrm{l}(\xkj),\\
            \bar{y}_{L_i} &= y_{L_i}^\ast + \rho_k (R_i \xkj - \ellRi) \geq 0, &&\mathrm{for~} i \in \mathcal{L}^\mathrm{l}(\xkj), \label{eq:duals:proof:single:solve:L}\\
            \bar{y}_{R_i} &= y_{R_i}^\ast + \rho_k (L_i \xkj - \ellLi) \geq 0, &&\mathrm{for~} i \in \mathcal{R}^\mathrm{l}(\xkj), \label{eq:duals:proof:single:solve:R}
        \end{align}
    \end{subequations}
    where $\rho_k$ is assumed to be large enough to satisfy the inequalities in~\eqref{eq:duals:proof:single:solve:L} and~\eqref{eq:duals:proof:single:solve:R}. 
    
    The existence of such a $\rho_k$ can be seen as follows. Applying \theoref{theo:penalty:convergence} shows that there certainly exists such a penalty parameter if $\xkj$ is replaced by $x^\ast$, and let us denote an adequate choice by~$\tilde{\rho}_k$. Let us assume $\xkj \neq x^\ast$ and that their active sets of the complementarity pairs coincide. Then $L_i \xkj - \ellLi > 0$ iff $L_i x^\ast - \ellLi > 0$ and $R_i \xkj - \ellRi > 0$ iff $R_i x^\ast - \ellRi > 0$. In such cases we may express 
    \begin{equation}
        L_i \xkj - \ellLi = \xi_i^L (L_ix^\ast - \ellLi) > 0, \quad R_i \xkj - \ellRi = \xi_i^R (R_ix^\ast - \ell_{R_i}) > 0
    \end{equation}
    for some $\xi_i^L,\xi_i^R > 0$, and choose $\rho_k = \max_i\{\xi_i^L, \xi_i^R\} \tilde{\rho}_k$. Thus the existence of a penalty parameter large enough to satisfy~\eqref{eq:duals:proof:single:solve} transfers immediately from~\theoref{theo:penalty:convergence}.

    We now want to show that $\xkj^\ast = x^\ast$. Note that the dual variables $(\bar{y}_A, \bar{y}_L, \bar{y}_R)$ respect the sign conditions required for a KKT point of~\eqref{min:LCQP:pen:inner}. Strong stationarity of $(x^\ast, y^\ast)$ yields
    \begin{subequations}
        \begin{align}
            0 &= Qx^\ast + g - A^\top y_A^\ast - L^\top y_L^\ast - R^\top y_R^\ast \\
            &= Qx^\ast + g - A^\top \bar{y}_A - L^\top(\bar{y}_L - \rho_k (R \xkj - \ell_R)) - R^\top (\bar{y}_R - \rho_k (L \xkj - \ell_L)) \\
            &= Qx^\ast + g + \rho_k ((L^\top R + R^\top L) \xkj + g_\varphi)- A^\top \bar{y}_A - L^\top \bar{y}_L - R^\top \bar{y}_R \\
            &= Qx^\ast + g_k + \rho_k C \xkj - A^\top \bar{y}_A - L^\top \bar{y}_L - R^\top \bar{y}_R,\label{eq:outer:loop:stationarity}
        \end{align}
    \end{subequations}
    where we recall $g_\varphi = - (L^\top \ell_R + R^\top \ell_L)$ and $g_k = g + \rho_k g_\varphi$. This shows that $(x^\ast, \bar{y}_A, \bar{y}_L, \bar{y}_R)$ is the unique KKT point of the strictly convex inner loop problem~\eqref{min:LCQP:pen:inner}. Thus $x^\ast = \xkj^\ast$.

    It remains to be shown that the globalization scheme will not interfere. More precisely, we must show $\alphakj^\ast = 1$. Note that $\varphi(\xkj + \alphakj \pkj) = 0$ for all $0 \leq \alphakj \leq 1$, as we never leave the complementarity satisfying active sets $\mathcal{L}^\mathrm{l}(\xkj) = \mathcal{L}^\mathrm{l}(x^\ast)$ and $\mathcal{R}^\mathrm{l}(\xkj) = \mathcal{R}^\mathrm{l}(x^\ast)$ (recall that the respective constraints are linear). Consequently, $\varphi(\xkj)$ can not be curved along $\pkj$, i.e.,
    \begin{equation}
        \pkj^\top C \pkj = 0.
    \end{equation}
    This leads to a full step $\alphakj^\ast = 1$ according to the step length formula~\eqref{eq:step:length:formula}. Finally
    \begin{equation}
        x_{k,j+1} = \xkj + \alphakj^\ast (\xkj^\ast - \xkj) = \xkj^\ast = x^\ast.
    \end{equation}
    Note that~\eqref{eq:outer:loop:stationarity} provides stationarity of the outer loop. The termination conditions of the outer loop are thus satisfied as well.
    \myqed
\end{proof}

\section{Numerical Experiments} \label{sec:testing}
\begin{table}
  \centering
  \ra{1.3}
  \caption{Number of variables, constraints and complementarity constraints listed for each problem test set. In this table we denote the mean value of a vector $x \in \R^n$ by $\bar{x}$.}
  \begin{tabular}{@{}lcccccc@{}}\toprule
      ~Problem Set & $\mathrm{range}(n)$ & $\widebar{n}$ & $\mathrm{range}(n_A)$ &
      $\widebar{n}_A$ & $\mathrm{range}(n_c)$ & $\widebar{n}_c$ \\ \midrule
      ~MacMPEC       & $(2,~2000)$ & $180$ & $(0,~1500)$ & $97$ & $(1,~1000)$ & $81$ \\
      ~IVOCP         & $(151,~301)$ & $226$ & $(50,~100)$ & $75$ & $(100,~200)$ & $150$ \\
      ~Moving Masses & $(554,~1104)$ & $829$ & $(034,~604)$ & $454$ & $(200,~400)$ & $300$ \\  \bottomrule
      %~Random LCQP   & & & & & & \\ \bottomrule
  \end{tabular}
  \label{tab:benchmark:ranges}
\end{table}
We consider three different benchmarks: the LCQP subset of the MacMPEC benchmark~\cite{leyffer2000MacMPEC}; an initial value \ac{ocp} of a discontinuous dynamic with a single switch; and an \ac{ocp} with the goal of bringing a system of moving masses to a steady state, in which the complementarities arise from a Coulomb friction model.
\tabref{tab:benchmark:ranges} gives an overview of the variable, constraint and complementarity dimensions of each problem set. The benchmarks are available at
\[ \href{https://github.com/hallfjonas/LCQPTest}{\texttt{https://github.com/hallfjonas/LCQPTest}} \]
and we encourage readers to reproduce the outcomes.

For each benchmark we compare the performance profile as introduced in~\cite{dolan2002benchmarking}. The used performance metric is the fraction of problems solved within a time factor $\tau$ compared to the fastest solver of each problem. Let $\mathcal{P}$ denote the set of problems of a given benchmark. Then, for a solver $s$, this fraction is denoted by
\begin{equation*}
  P_s = \frac{\{p \in \mathcal{P} : r_{p,s} \leq \tau\}}{| \mathcal{P} |},
\end{equation*} 
where the ratio $r$ is defined by 
\begin{equation*}
    r_{p,s} = \frac{\mathrm{CPU~time~of~}s\mathrm{~to~solve~}p}{\mathrm{fastest~solver~CPU~time~to~solve~}p}.
\end{equation*}
On failed convergence, we set this value to $\infty$.

We consider various methods: \solverref\ with qpOASES~\cite{Ferreau2014}; \solverref\ with qpOASES utilizing the sparse linear solver MA57~\cite{ma57}; \solverref\ with OSQP~\cite{osqp}; Gurobi~\cite{gurobi2021} for solving the \ac{miqp} reformulation; IPOPT~\cite{wachter2006implementation} for solving the respective problems~\eqref{min:LCQP:pen},~\eqref{min:LCQP:nlp:smooth},~\eqref{min:LCQP:nlp:reg}; and finally IPOPT NLP, which solves~\eqref{min:LCQP:nlp:reg} for a sufficiently small $\sigma > 0$ without the homotopy procedure. The solver IPOPT is called through its CasADi interface~\cite{Andersson2018}. We carefully chose to use internal solver timings in order to remove as much overhead as possible, though some timings may still include some amount of overhead. We remark that the timings including the overhead yield very similar results; mostly the \ac{miqp} method varies as its models are rebuilt by Gurobi. 

\subsection{The LCQP subset of MacMPEC}
The MacMPEC problem set~\cite{leyffer2000MacMPEC} contains a variety of optimization problems with complementarity constraints, out of which we selected the ones fitting the \ac{lcqp} framework, i.e., the ones that have a convex quadratic objective function and linear constraints. The set contains a range of small to large problems, though the vast majority of the problems have small dimensions. The resulting subset contains 39 problems modeled in AMPL~\cite{fourer2003ampl}. In order to run the problems with our solver we translated the examples into Matlab.

\begin{figure}
  \centering
  \includegraphics[width=\linewidth,center]{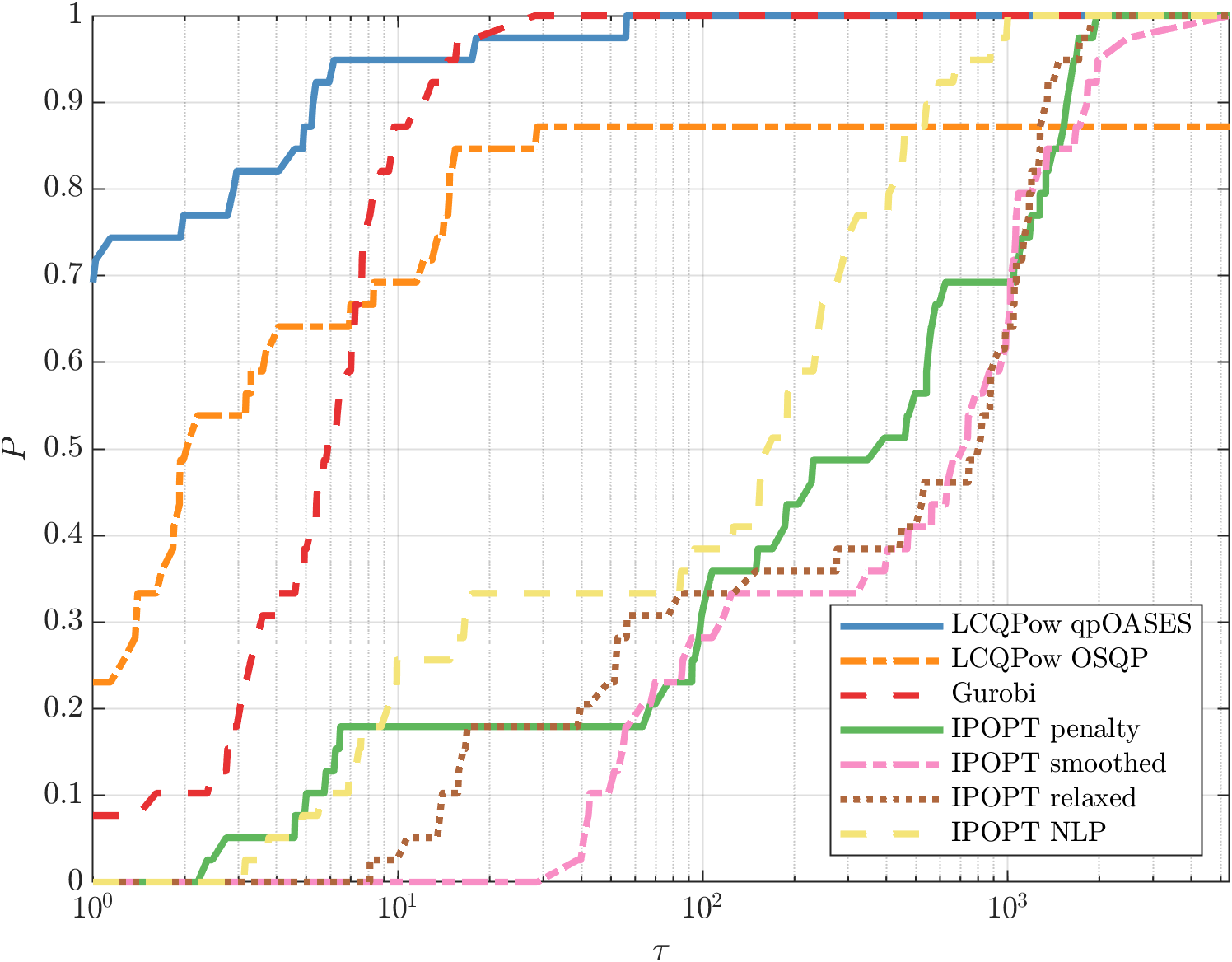}
  \caption{Performance plot comparing various solution variants for the \ac{lcqp} subset of MacMPEC~\cite{leyffer2000MacMPEC}.}
  \label{fig:benchmark:MacMPEC:performance}
\end{figure}
\figref{fig:benchmark:MacMPEC:performance} shows that \solverref\ with qpOASES mostly outperforms all other methods. Tables \ref{tab:MacMPEC:LCQPow:Gurobi} and \ref{tab:MacMPEC:IPOPT} show the objective values obtained by the various methods for each problem. On top of efficient solution computation, this supports that our method is able to find the global solution for many of the posed problems. However, for some problems the solver gets stuck in local solutions. In this benchmark, \solverref\ with OSQP also achieves fast results for many problems, however, it is less robust and convergence fails for about a quarter of the problems.

\subsection{Initial Value Problem}\label{sec:experiments:IVOCP}
We now consider an initial value finding problem which was introduced by Stewart and Anitescu~\cite[Section 2]{stewart2010optimal}. This numerical test example contains a dynamical system with a discontinuous right hand side, in which a single switch occurs. The position of the switch is solely dependent on the initial value. The optimization problem in continuous time is given by
\begin{mini!}
	{x_0 \in \R,~x(\cdot) \in \mathcal{C}^0}{\int_0^2 x(t)^2 \text{d}t + (x(2) - 5/3)^2}
	{\label{eq:IVOCP}}{}
    \addConstraint{x(0)}{= x_0, \label{eq:IVOCP:initial:value}}
    \addConstraint{\dot{x}(t)}{\in 2 - \sign (x(t)),}{\quad t \in [0, 2] \label{eq:IVOCP:dynamics}}.
\end{mini!}
The discontinuous dynamics~\eqref{eq:IVOCP:dynamics} describe a \ac{fdi} and can be reformulated into a linear complementarity system~\cite{nurkanovic2020limits}. We then discretize the system using the implicit Euler scheme with $N$ nodes. The resulting \ac{lcqp} reads as
\begin{mini!}
  {\substack{x_0,\dots,x_N \in \R \\ y_0,\dots,y_{N-1} \in \R \\ \lambda^-_0,\dots,\lambda^-_{N-1} \in \R}}{\sum_{k=0}^{N-1} h x_k^2 + \left(x_N-\frac{5}{3}\right)^2}
	{\label{eq:IVOCP:LCQP}}{}
    \addConstraint{x_k - x_{k-1} - h \bigl(3(1 - y_k) + y_k \bigr)}{= 0, \qquad \mathrm{for~}k=1,\dots,N,}
    \addConstraint{0 \leq x_k + \lambda^-_k \perp 1 - y_k }{\geq 0, \qquad \mathrm{for~}k=1,\dots,N,}
    \addConstraint{0 \leq \lambda^-_k \perp y_k }{\geq 0, \qquad \mathrm{for~}k=1,\dots,N,}
\end{mini!}
where $h = T/N$ is the discretization step size. The benchmark is created by varying $N \in \{50, 55, \dots, 100\}$ and the initial guess $x_0 \in \mathcal{X}_0$ for the initial value, where $\mathcal{X}_0$ contains $10$ equidistant values between $-1.9$ and $-0.9$.

\figref{fig:benchmark:IVOCP:timings} presents the performance profile comparing various methods, showing that~\solverref\ outperforms the other methods; particularly the OSQP variant achieves fast results. This significant speed-up does not suffer from a trade-off in terms of solution quality, as it is able to find the same solutions as the \ac{miqp} reformulation as shown in~\figref{fig:benchmark:IVOCP:objective}. Most homotopy approaches solved via IPOPT find the same solutions in this benchmark. The NLP reformulation achieves convergence with similar speed compared to the qpOASES variant of~\solverref, however, its solution quality highly depends on the initialization.

\begin{figure}
  \centering
  \begin{subfigure}[b]{.9\linewidth}
    \includegraphics[width=\textwidth]{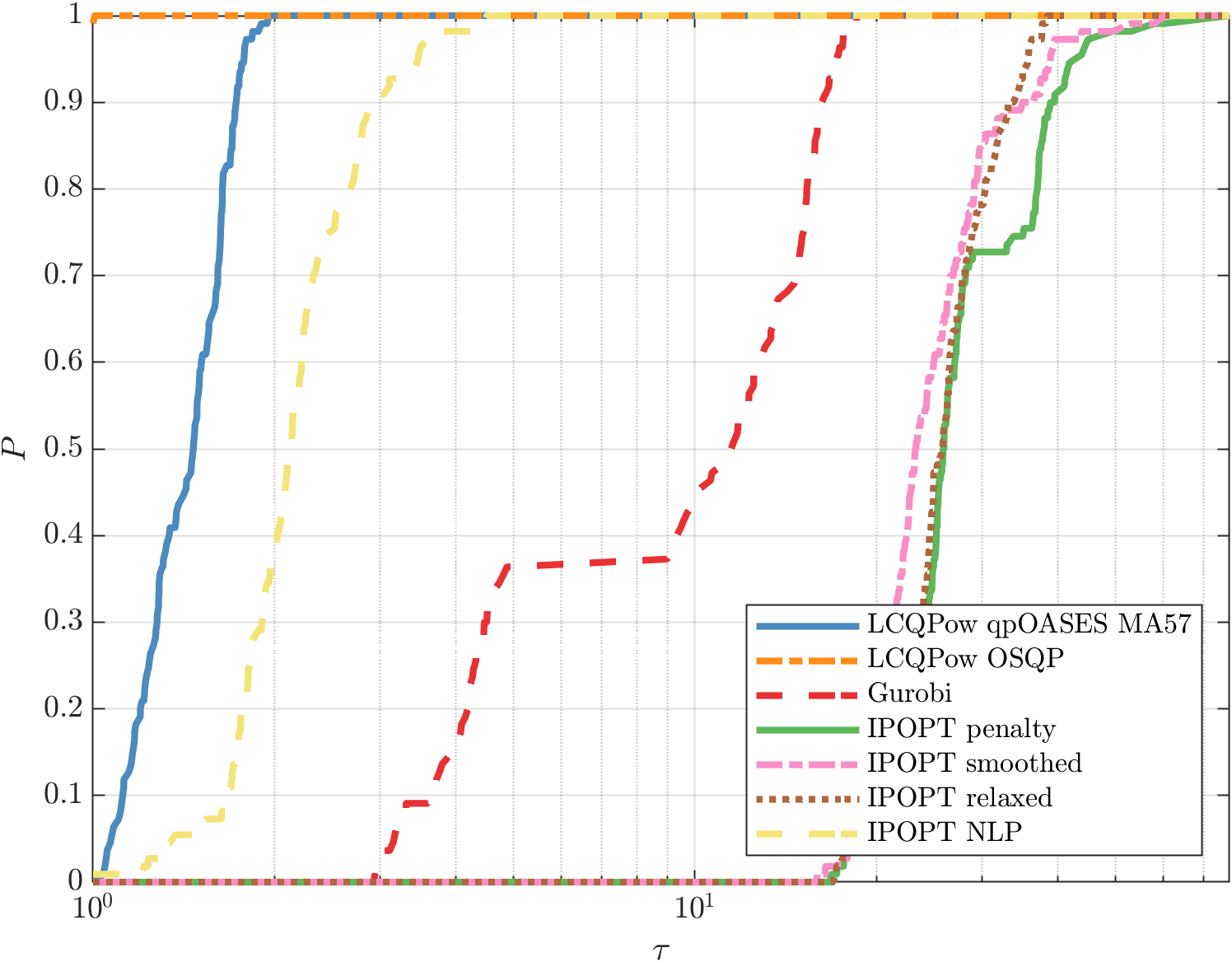}
    \caption{timings}
    \label{fig:benchmark:IVOCP:timings}
  \end{subfigure} \\%  
  \begin{subfigure}[b]{.9\linewidth}
    \includegraphics[width=\textwidth]{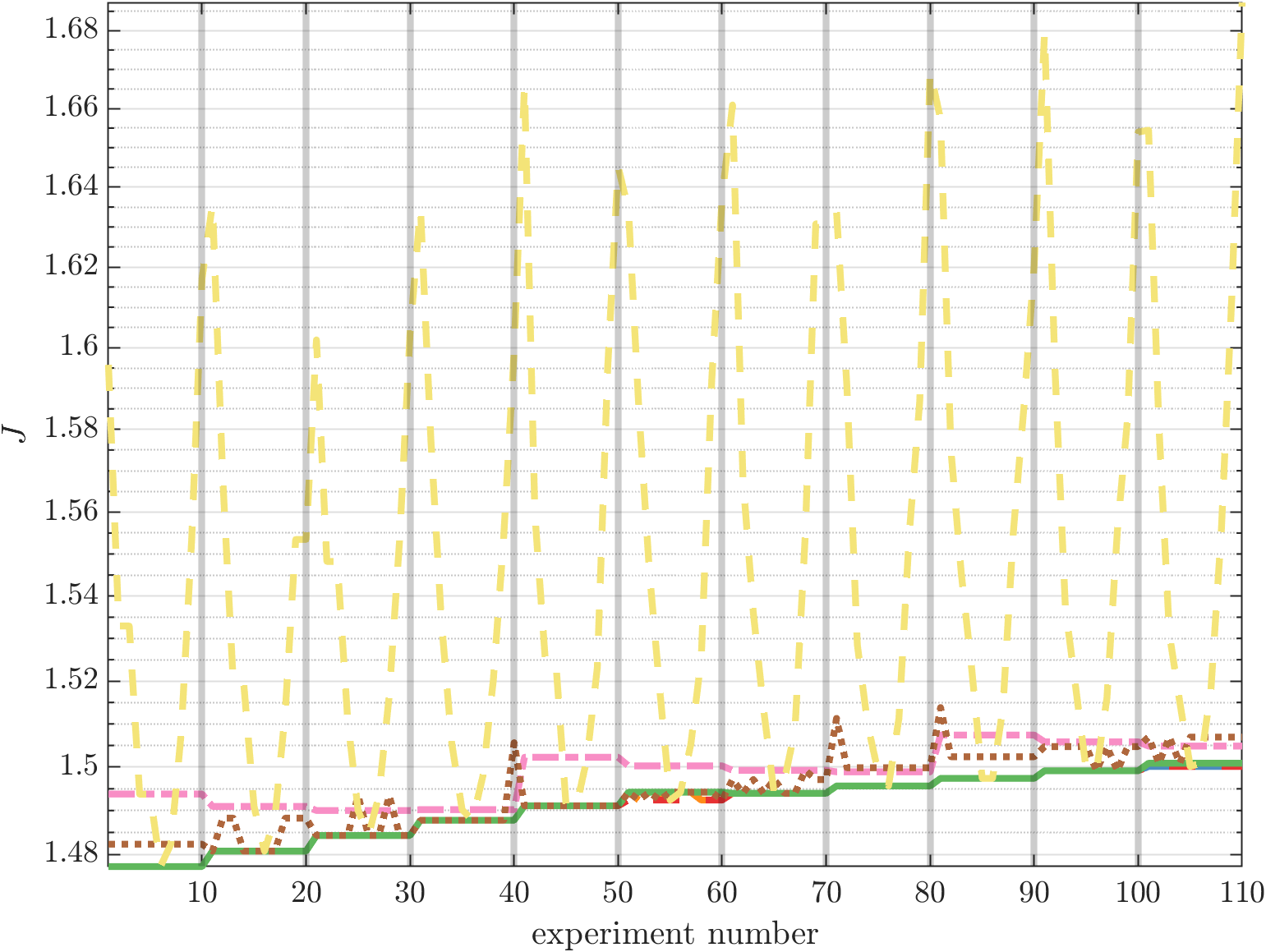}
      \caption{objective}
      \label{fig:benchmark:IVOCP:objective}
  \end{subfigure}%
  \caption{Performance profile and objective function comparison for the Initial Value \ac{ocp}. The vertical lines on the right indicate a change in the number of discretization intervals from $N=50$ (experiments 1 through 10) up to $N=100$ (experiments 101 through 110). For a fixed $N$, the initial guess for $x_0$ is varied from $-1.9$ (left most experiment) to $-0.9$ (right most experiment).}
  \label{fig:benchmark:IVOCP}
\end{figure}

\subsection{Moving Masses}\label{sec:experiments:MovingMasses}
We present an OCP formulation of the test problem described by Stewart in~\cite[Section 5]{stewart1996numerical}. Consider a number of $s$ springs connecting $s$ masses: the first mass $M_1$ is connected to a wall and each other mass $M_i$ is connected to its preceding mass $M_{i-1}$. We assume that the rest length of the spring has no influence in the dynamics and that the positions of the masses are given in their own coordinate frame. The position $p_i$ of mass $M_i$ is thus $0$ if the spring attached left to the mass $M_i$ is relaxed, and we assume that the masses never collide. We introduce a control $u \in \R$, which represents a force applied to the last mass $M_s$. The states are described by $x = (p, v) \in \R^{2s}$, where $p$ and $v$ capture the positions and velocities of the respective masses. The full setup is depicted in \figref{fig:MovingMasses:SetUp}.
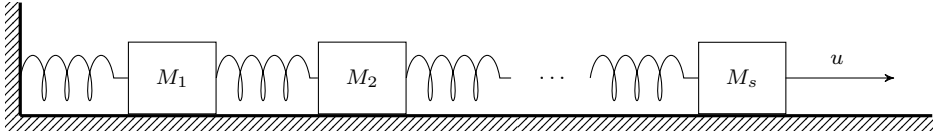
\begin{figure}
  \centering
  %%
%% Moving masses with frictional forces
%%

\usetikzlibrary{patterns}

\begin{tikzpicture}
\draw[very thick] (0,0) -- (0,1.5);
\draw[very thick] (0,0) -- (12,0);
\fill [pattern = north east lines] (-0.2,0) rectangle (0,1.5);
\fill [pattern = north east lines] (-0.2,-0.2) rectangle (12,0);
\node at (2,0.5) [rectangle,draw,inner sep=3.6mm] (a) {$M_1$};
\draw[decoration={aspect=0.3, segment length=3mm, amplitude=3mm,coil},decorate]
(0,0.5) -- (a);

\node at (4.5,0.5) [rectangle,draw,inner sep=3.6mm] (b) {$M_2$};
\draw[decoration={aspect=0.3, segment length=3mm, amplitude=3mm,coil},decorate]
(2.58,0.5) -- (b);

\node at (7,0.5) [rectangle,inner sep=3.6mm] (c) {$\ldots$};
\draw[decoration={aspect=0.3, segment length=3mm, amplitude=3mm,coil},decorate]
(5.08,0.5) -- (c);

\node at (9.5,0.5) [rectangle,draw,inner sep=3.6mm] (d) {$M_s$};
\draw[decoration={aspect=0.3, segment length=3mm, amplitude=3mm,coil},decorate]
(7.5,0.5) -- (d);

\draw[->] (10.07,0.5) -- (11.5,0.5);

\node at (10.75,0.75) {$u$};
\end{tikzpicture}
  \caption{Setup of the masses and springs.}
  \label{fig:MovingMasses:SetUp}
\end{figure}

Each mass slides over the ground and introduces a frictional force. The direction of this force changes with a sign change of the respective velocity, and thus leads to discontinuous dynamics,
\begin{subequations}
  \begin{align}
    v(0) &= \bar{v}_0,\\
    p(0) &= \bar{p}_0, \\
    \dot{p} &= v, \\
    \dot{v}_i &\in F_i(x) = \begin{cases}
        ( -p_1) + (p_2 - p_1) - v_1 - 0.3 \cdot \sign(v_1), & i = 1, \\
        ( p_{i-1}-p_i) + (p_{i+1}-p_i) - v_i - 0.3 \cdot \sign(v_i), & 1 < i < s, \\
        ( p_s - p_{s-1}) - v_s - 0.3 \cdot \sign(v_s) + u, & i = s,
    \end{cases}
  \end{align}
\end{subequations}
where the initial value $\bar{x}_0 = (\bar{v}_0^\top,\bar{p}_0^\top)^\top \in \R^{2s}$ is fixed. Again, we reformulate the \ac{fdi} into the dynamic complementarity system
\begin{subequations}
  \begin{align}
    x(0) &= \bar{x}_0,\\
    \dot{p} &= v, \label{eq:MovingMasses:smooth:ODE:p}\\
    \dot{v}_i &= \begin{cases}
        ( -p_i) + (p_{i+1} - p_i) - v_1 - 0.3\cdot(2y_i(t) - 1), & i = 1, \\
        ( p_{i-1}-p_i) + (p_{i+1}-p_i) - v_i - 0.3\cdot(2y_i(t) - 1), & 1 < i < s, \\
        ( p_{i-1} - p_i) + u - v_i - 0.3\cdot(2y_i(t) - 1), & i = s,
    \end{cases} \label{eq:MovingMasses:smooth:ODE:v}\\
    0 &\leq \lambda^- + v \perp 1 - y \geq 0, \\
    0 &\leq \lambda^- \perp y \geq 0.
  \end{align}
\end{subequations}

\begin{figure}
  \centering
  \begin{subfigure}[b]{.5\linewidth}
    \raggedright
    \includegraphics[width=0.97\textwidth]{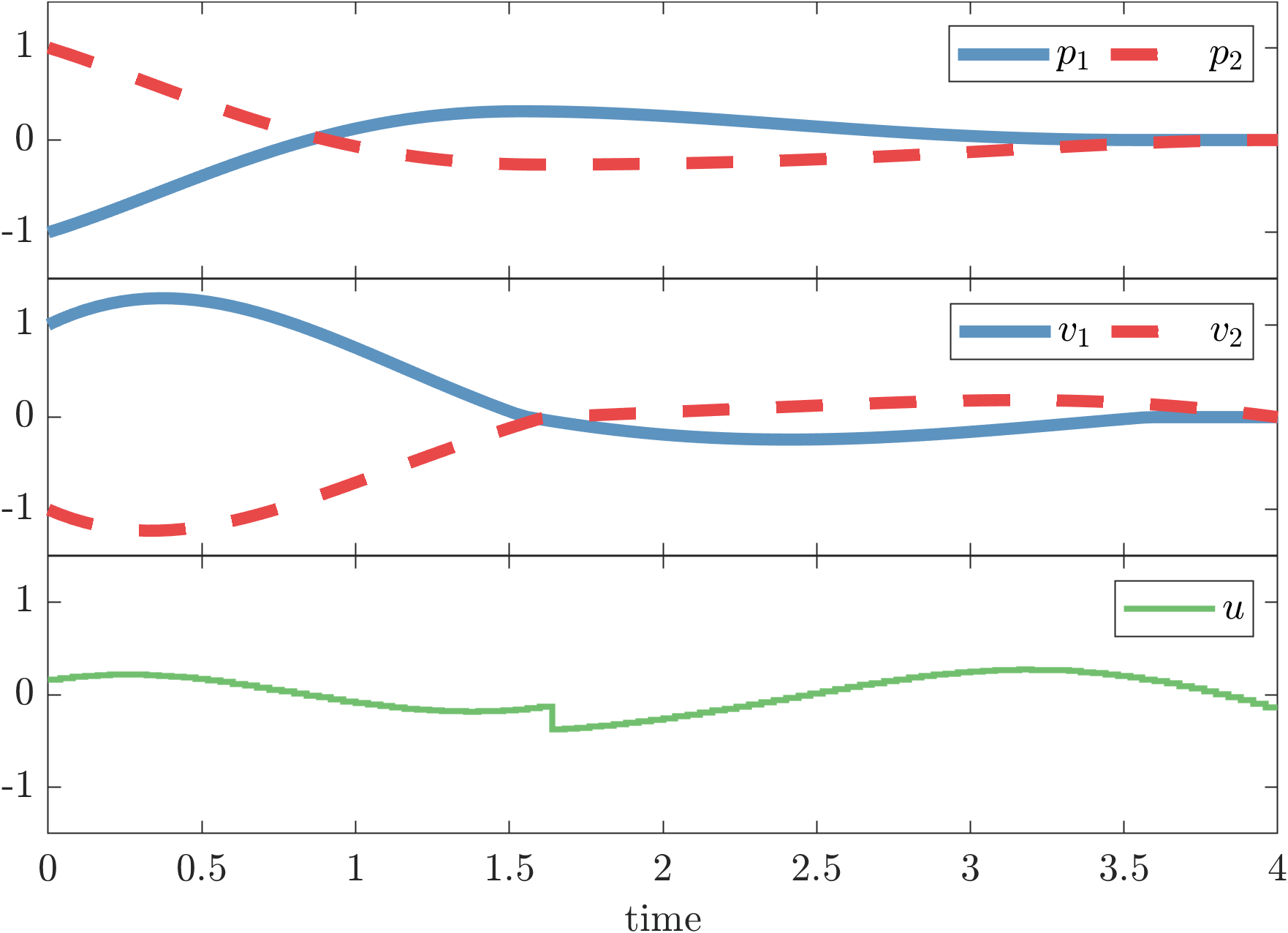}
  \end{subfigure}%  
  \begin{subfigure}[b]{.5\linewidth}
    \raggedleft
    \includegraphics[width=0.97\textwidth]{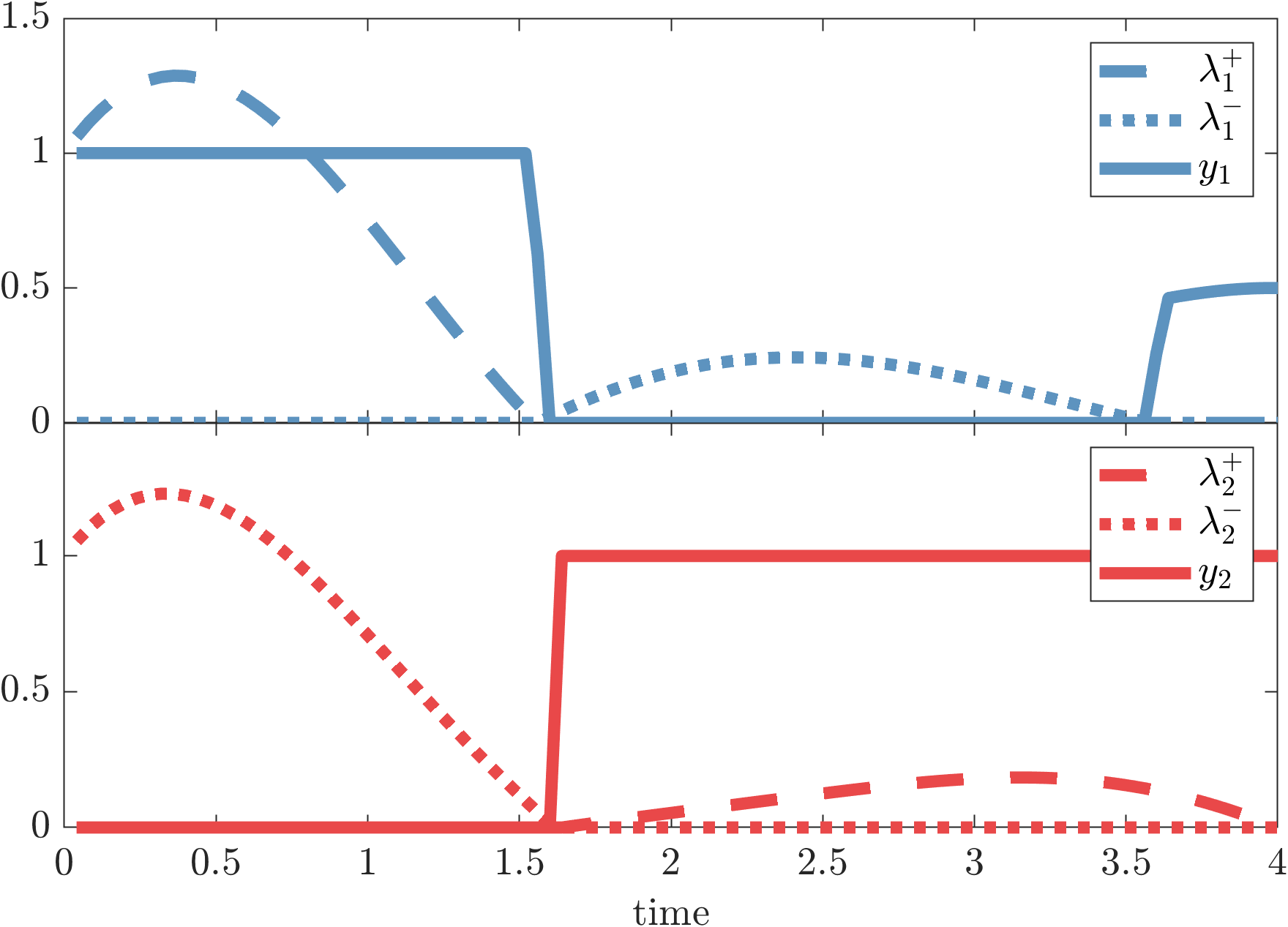}
  \end{subfigure}%
  \caption{Sample solution for $s = 2$ masses. The left plot shows the trajectories of the mass positions and velocities together with the control input. The right plot depicts the evolution of the algebraic states.}
  \label{fig:MovingMasses:trajcetories}
\end{figure}%

Let $f(x, y, \lambda, u)$ denote the right hand side of~\eqref{eq:MovingMasses:smooth:ODE:p}-\eqref{eq:MovingMasses:smooth:ODE:v} such that $\dot{x} = f(x,y,\lambda,u)$. We formulate the goal of forcing the system into the equilibrium point of the resting position, i.e., to obtain $x(T) = 0$, while penalizing the control input and equilibrium deviation at each stage. The \ac{ocp} in continuous time reads as
\begin{mini!}
  {x, u, y, \lambda^-}{\int_0^T x(t)^\top x(t) + u(t)^2 \textup{d}t}
	{\label{eq:MovingMasses}}{}
  \addConstraint{x(0)}{ = \hat{x}_0,}
  \addConstraint{\dot{x}(t)}{ = f(x(t), y(t), \lambda(t), u(t)),}
  \addConstraint{0}{\leq v(t) + \lambda^-(t) \perp 1 - y(t) \geq 0,}
  \addConstraint{0}{\leq \lambda^-(t) \perp y(t) \geq 0,}
  \addConstraint{0}{= x(T).}
\end{mini!}
This \ac{ocp} is again discretized using implicit Euler with $50 \leq N \leq 100$ nodes over a varying time range of $2 \leq T \leq 4$. A trajectory for a sample solution is shown in~\figref{fig:MovingMasses:trajcetories}. The performance profile for $s = 2$ masses is given in \figref{fig:MovingMasses:benchmark:timings}. \solverref\ with qpOASES exploiting sparsity with MA57 clears the benchmark fastest, though the relaxed and NLP variants solved via IPOPT achieve similar results. The OSQP variant of \solverref\ is less robust, as it only solves about $75\%$. However, the problems for which it succeeds are solved significantly faster than the other method. Due to the increased number of complementarity constraints, the \ac{miqp} variant is outperformed. \figref{fig:MovingMasses:benchmark:objective} shows that, up to a few exceptions, all methods find solutions of the same quality, which are most likely the same solutions.

\begin{figure}
  \centering
  \begin{subfigure}[b]{.9\linewidth}
    \includegraphics[width=\textwidth]{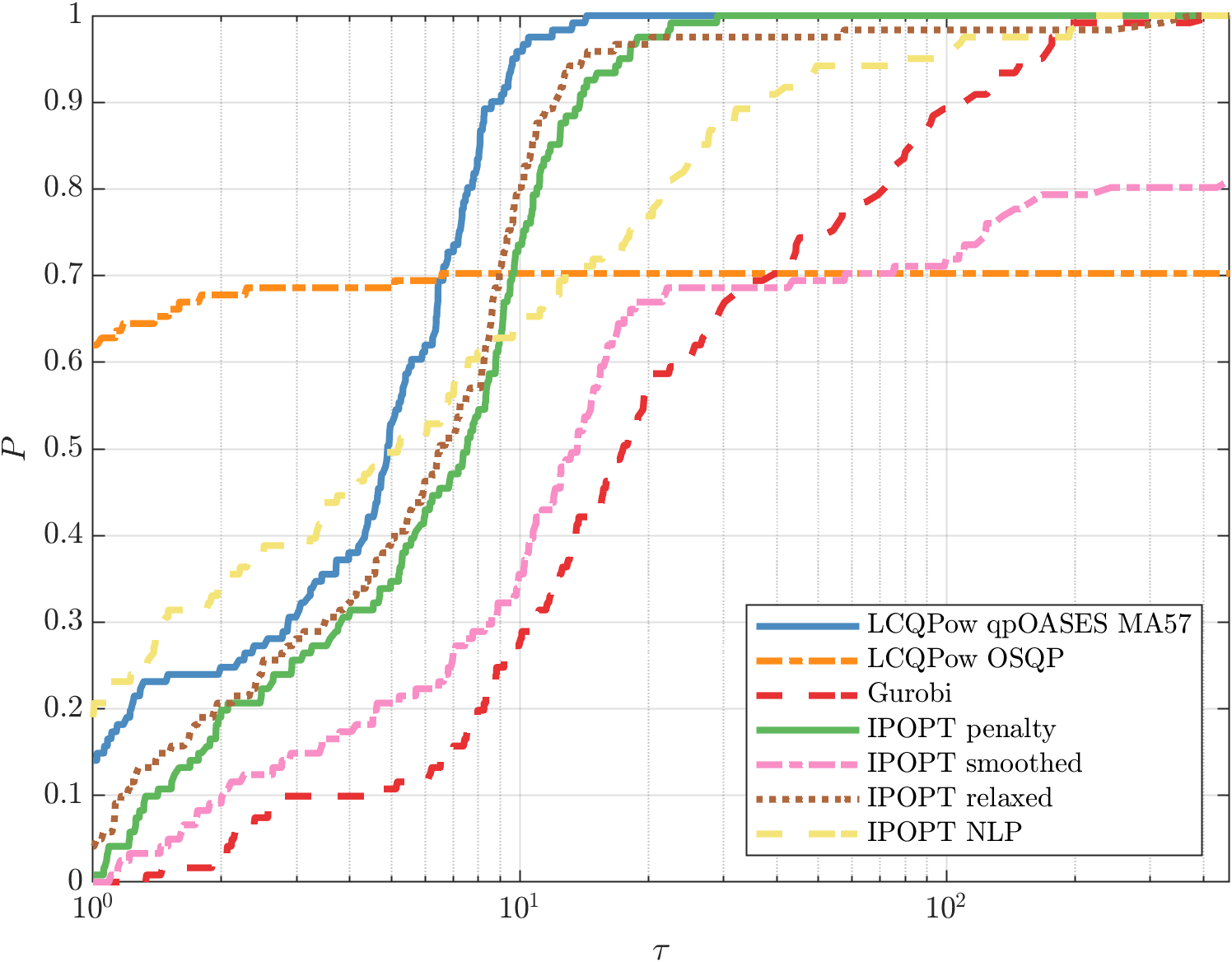}
    \caption{timings}
    \label{fig:MovingMasses:benchmark:timings}
  \end{subfigure} \\%  
  \begin{subfigure}[b]{.9\linewidth}
    \includegraphics[width=\textwidth]{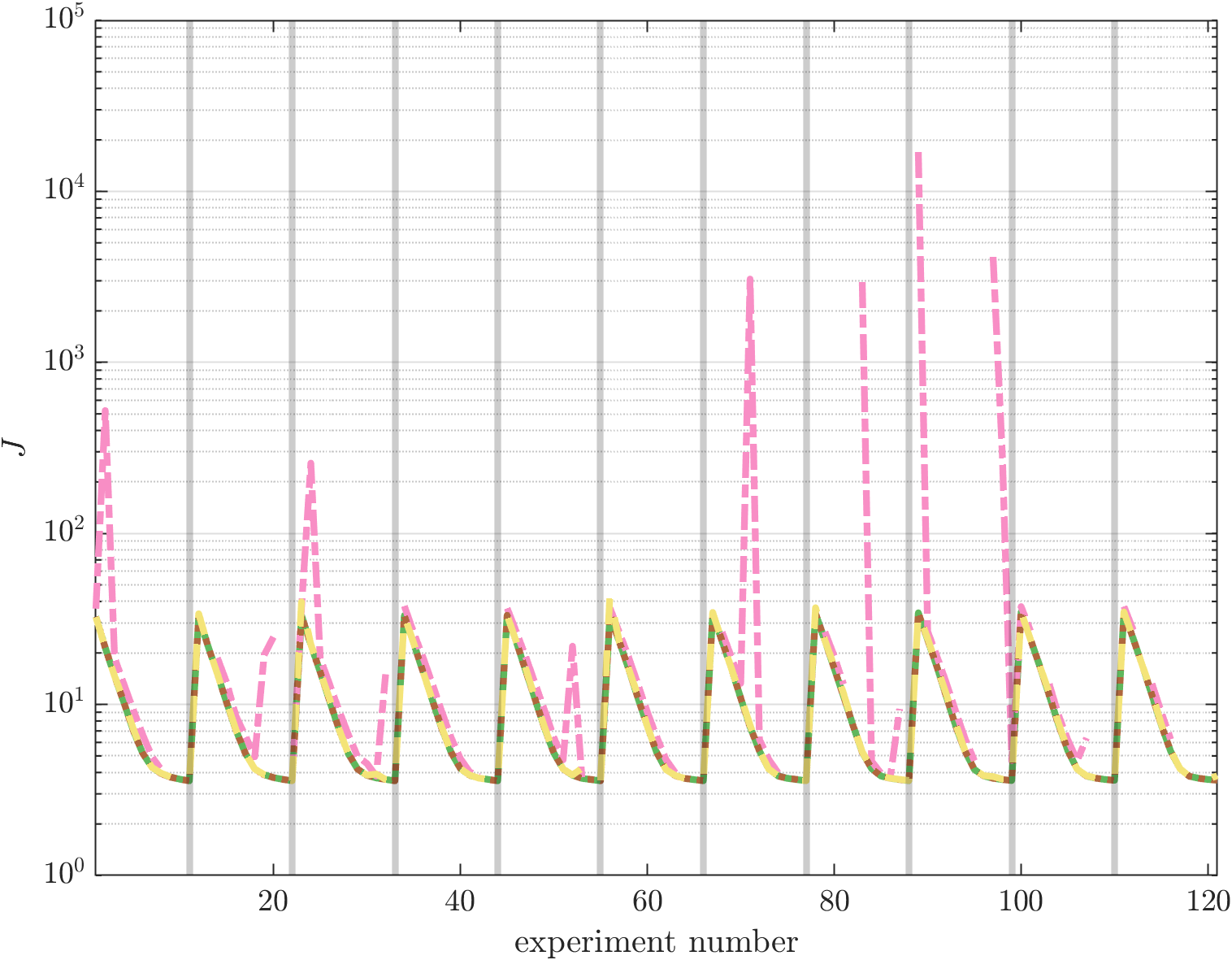}
      \caption{objective}
      \label{fig:MovingMasses:benchmark:objective}
  \end{subfigure}%
  \caption{Performance profile and objective function comparison for the Moving Masses \ac{ocp}. The vertical lines on the right indicate a change in the number of discretization intervals from $N=50$ (experiments 1 through 10) up to $N=100$ (experiments 101 through 110). For a fixed $N$, the experiment time $T$ is varied between 2 (left) and 4 (right).}
  \label{fig:MovingMasses:benchmark}
\end{figure}

% \section{Random LCQPs}
% \todo{do benchmark and include figure} shows that the presented algorithm outperforms all other solvers with respect to the average performance profile. IPOPT performs slowest in these problem sets, possibly due to some overhead from the CasADi interface call. However, IPOPT is the most robust, as it is able to solve all optimization problems except for one. In terms of robustness, LCQP is able to achieve the next best stability, solving two problems less than IPOPT. Utilizing the OSQP solver in the LCQP algorithm provides a similar performance profile as with qpOASES, however the robustness is weakened. This is due to the fact that OSQP solves the QP subproblems with less precision, which has a disadvantageous affect on the inner loop convergence.
%\input{chapters/limitations.tex}
%\input{chapters/miqp.tex}
\section{Conclusion}\label{sec:summary}
In conclusion, the introduced solver \solverref\ was demonstrated to provide competitive solutions for quadratic programming problems with linear complementarity constraints. It offers user flexibility, including the choice of inherently different QP solvers on the lower level, that allow choosing between the tradeoff of robustness and high performance.

\newcommand{\rowspacing}{1.3}
\newcommand{\headerformats}{lrrrr}
\newcommand{\headers}{problem & best known & LCQPow qpOASES & LCQPow OSQP & Gurobi \
}

\ra{\rowspacing}
\begin{table}
    \centering
    \caption{Objective values for MacMPEC solutions obtained via \solverref~and Gurobi.}
    \label{tab:MacMPEC:LCQPow:Gurobi}
    \begin{tabular}{@{}\headerformats @{}}\toprule 
    \headers \\ \midrule
    bard1 & 17 & 25 & - & 17\\
    bard1m & 17 & 25 & - & 17\\
    bard2 & -6598 & -6598 & -6598 & -6598\\
    bilevel2 & -6600 & -6600 & -6600 & -6600\\
    bilevel2m & -6600 & -6600 & -6600 & -6600\\
    ex9.2.1 & 17 & 25 & - & 17\\
    ex9.2.2 & 100 & 1.00e+02 & 1.00e+02 & 100\\
    ex9.2.4 & 5.00e-01 & 5.00e-01 & 5.00e-01 & 5.00e-01\\
    ex9.2.5 & 5 & 9 & 9 & 5\\
    ex9.2.6 & -1 & -1 & -1 & -1\\
    ex9.2.7 & 17 & 25 & - & 17\\
    flp2 & 0 & 2.34e-12 & 2.34e-12 & 2.34e-12\\
    flp4.1 & 0 & -1.55e-15 & -1.55e-15 & 0\\
    flp4.2 & 0 & -2.22e-16 & -1.11e-15 & 0\\
    flp4.3 & 0 & 1.55e-15 & 1.55e-15 & 0\\
    flp4.4 & 0 & 6.66e-16 & -2.22e-15 & 0\\
    gauvin & 20 & 20 & 20 & 20\\
    hs044.i & 1.56e+01 & 6.25e-06 & - & 6.25e-06\\
    jr1 & 5.00e-01 & 5.00e-01 & 5.00e-01 & 5.00e-01\\
    jr2 & 5.00e-01 & 5.00e-01 & 5.00e-01 & 5.00e-01\\
    kth2 & 0 & -2.22e-16 & -2.22e-16 & 0\\
    kth3 & 5.00e-01 & 5.00e-01 & 5.00e-01 & 5.00e-01\\
    liswet1.050 & 1.40e-02 & 1.40e-02 & 1.40e-02 & 1.40e-02\\
    liswet1.100 & 1.37e-02 & 1.37e-02 & - & 1.37e-02\\
    liswet1.200 & 1.70e-02 & 1.70e-02 & - & 3.38e-02\\
    nash1a & 7.89e-30 & 5.20e-12 & 5.20e-12 & 4.71e-12\\
    nash1b & 7.89e-30 & 5.20e-12 & 5.20e-12 & 4.71e-12\\
    nash1c & 7.89e-30 & 5.20e-12 & 5.20e-12 & 4.71e-12\\
    nash1d & 7.89e-30 & 5.20e-12 & 5.20e-12 & 4.71e-12\\
    nash1e & 7.89e-30 & 5.20e-12 & 5.20e-12 & 4.71e-12\\
    portfl1 & 1.50e-05 & 2.04e-02 & 2.04e-02 & 2.04e-02\\
    portfl2 & 1.46e-05 & 2.78e-02 & 2.78e-02 & 2.79e-02\\
    portfl3 & 6.27e-06 & 2.28e-02 & 2.27e-02 & 2.31e-02\\
    portfl4 & 2.18e-06 & 2.05e-02 & 2.05e-02 & 2.21e-02\\
    portfl6 & 2.36e-06 & 2.40e-02 & 2.40e-02 & 2.16e-01\\
    qpec1 & 80 & 80 & 80 & 80\\
    qpec2 & 45 & 4.50e+01 & 4.50e+01 & 45\\
    scholtes3 & 5.00e-01 & 5.00e-01 & 5.00e-01 & 5.00e-01\\
    sl1 & 1.00e-04 & 1.00e-04 & 1.00e-04 & 1.00e-04\\
    \bottomrule
    \end{tabular}
\end{table}

\renewcommand{\headerformats}{lrrrrr}
\renewcommand{\headers}{problem & best known & penalty & smoothed & relaxed & NLP\
}

\ra{\rowspacing}
\begin{table}
    \centering
    \caption{Objective values for MacMPEC solutions obtained by the IPOPT variants.}
    \label{tab:MacMPEC:IPOPT}
    \begin{tabular}{@{}\headerformats @{}}\toprule 
    \headers \\ \midrule
    bard1 & 17 & 17 & 17 & 25 & 17\\
    bard1m & 17 & 17 & 17 & 25 & 17\\
    bard2 & -6598 & -6598 & -6.60e+03 & -6598 & -6598\\
    bilevel2 & -6600 & -6600 & -6600 & -6600 & -6600\\
    bilevel2m & -6600 & -6600 & -6600 & -6600 & -6600\\
    ex9.2.1 & 17 & 17 & 17 & 25 & 17\\
    ex9.2.2 & 100 & 1.00e+02 & 1.00e+02 & 1.00e+02 & 1.00e+02\\
    ex9.2.4 & 5.00e-01 & 5.00e-01 & 5.00e-01 & 5.00e-01 & 5.00e-01\\
    ex9.2.5 & 5 & 9 & 5 & 9 & 9.80e+00\\
    ex9.2.6 & -1 & -1 & -1.00e+00 & -1 & -1\\
    ex9.2.7 & 17 & 17 & 17 & 25 & 17\\
    flp2 & 0 & 2.34e-12 & 8.87e-12 & 2.34e-12 & 7.60e-12\\
    flp4.1 & 0 & -2.70e-09 & 7.02e-06 & -2.70e-09 & -2.70e-09\\
    flp4.2 & 0 & -5.40e-09 & 6.32e-06 & -5.40e-09 & -5.40e-09\\
    flp4.3 & 0 & -6.30e-09 & 7.90e-06 & -6.30e-09 & -6.30e-09\\
    flp4.4 & 0 & -9.00e-09 & 1.18e-05 & -9.00e-09 & -9.00e-09\\
    gauvin & 20 & 20 & 325 & 20 & 20\\
    hs044.i & 1.56e+01 & 6.25e-06 & 1.56e+01 & 6.25e-06 & 1.56e+01\\
    jr1 & 5.00e-01 & 5.00e-01 & 5.00e-01 & 5.00e-01 & 5.00e-01\\
    jr2 & 5.00e-01 & 5.00e-01 & 5.00e-01 & 5.00e-01 & 5.00e-01\\
    kth2 & 0 & -9.00e-11 & 1.00e-07 & -9.00e-11 & -8.72e-11\\
    kth3 & 5.00e-01 & 5.00e-01 & 5.00e-01 & 5.00e-01 & 5.00e-01\\
    liswet1.050 & 1.40e-02 & 1.40e-02 & 1.30e-01 & 1.40e-02 & 1.40e-02\\
    liswet1.100 & 1.37e-02 & 1.37e-02 & 2.40e-01 & 1.37e-02 & 1.37e-02\\
    liswet1.200 & 1.70e-02 & 1.70e-02 & 4.75e-01 & 1.70e-02 & 1.70e-02\\
    nash1a & 7.89e-30 & 1.27e-11 & 4.71e-12 & 5.59e-12 & 4.71e-12\\
    nash1b & 7.89e-30 & 1.25e-11 & 4.71e-12 & 5.59e-12 & 4.71e-12\\
    nash1c & 7.89e-30 & 1.25e-11 & 4.71e-12 & 5.59e-12 & 4.71e-12\\
    nash1d & 7.89e-30 & 1.27e-11 & 4.71e-12 & 5.59e-12 & 4.71e-12\\
    nash1e & 7.89e-30 & 5.55e-12 & 4.71e-12 & 5.59e-12 & 4.71e-12\\
    portfl1 & 1.50e-05 & 2.04e-02 & 2.06e-02 & 2.04e-02 & 2.04e-02\\
    portfl2 & 1.46e-05 & 2.78e-02 & 2.97e-02 & 2.78e-02 & 2.78e-02\\
    portfl3 & 6.27e-06 & 2.27e-02 & 2.27e-02 & 2.27e-02 & 2.28e-02\\
    portfl4 & 2.18e-06 & 2.05e-02 & 2.06e-02 & 2.05e-02 & 2.05e-02\\
    portfl6 & 2.36e-06 & 2.40e-02 & 2.46e-02 & 2.40e-02 & 2.40e-02\\
    qpec1 & 80 & 80 & 8.00e+01 & 80 & 80\\
    qpec2 & 45 & 4.50e+01 & 4.50e+01 & 4.50e+01 & 4.50e+01\\
    scholtes3 & 5.00e-01 & 5.00e-01 & 5.00e-01 & 5.00e-01 & 1.00e+00\\
    sl1 & 1.00e-04 & 1.00e-04 & 1.00e-04 & 1.00e-04 & 1.00e-04\\
    \bottomrule
    \end{tabular}
\end{table}

%\listoftodos

\newpage
\bibliographystyle{ieeetr}
\bibliography{bibtex/biblio}

\end{document}